\documentclass[manuscript]{biometrika}

\usepackage{times}
\usepackage{bm}
\usepackage{natbib}
\usepackage{amssymb}
\usepackage{amsmath}
\usepackage{color}

\def\nn{{\nonumber}}

\def\Ind{{1\!\mathrm{l}}}
\DeclareMathOperator{\Tr}{tr}
\DeclareMathOperator{\diag}{diag}
\newcommand{\eig}{\mathrm{eig}}

\newcommand{\natz}{\mathbb{N}_0}
\newcommand{\inte}{\mathbb{Z}}

\newcommand{\diam}{{\rm diam}}

\newcommand{\pr}{\mathrm{pr}}

\renewcommand{\S}{\mathcal{S}}
\renewcommand{\O}{\mathcal{O}}

\newcommand{\scQ}{\mathcal{Q}}

\newcommand{\alphamax}{\alpha_{\max}}

\newcommand{\dstar}{d^*}
\begin{document}

%\jname{Biometrika}
%% The year, volume, and number are determined on publication
%\jyear{2011}
%\jvol{98}
%\jnum{1}
%% The \doi{...} and \accessdate commands are used by the production team
%\doi{10.1093/biomet/asm023}
%\accessdate{Advance Access publication on 30 June 2011}
%\copyrightinfo{\Copyright\ 2011 Biometrika Trust\goodbreak {\em Printed in Great Britain}}

%% These dates are usually set by the production team
%\received{January 2011}
%\revised{June 2011}

%% The left and right page headers are defined here:
\markboth{W. Shen, S.T. Tokdar \and S. Ghosal}{Bayesian density estimation}

%% Here are the title, author names and addresses
\title{Adaptive Bayesian multivariate density estimation with Dirichlet mixtures}

\author{Weining Shen}
\affil{Department of Statistics, North Carolina State University, 5109 SAS Hall, Campus Box 8203, Raleigh, North Carolina 27695, USA \email{wshen2@ncsu.edu} }
\author{Surya T. Tokdar}
\affil{Department of Statistical Science, Duke University, 219A Old Chemistry Building, Campus Box 90251, Durham, North Carolina 27708, USA
\email{tokdar@stat.duke.edu} }
\author{\and Subhashis Ghosal}
\affil{Department of Statistics, North Carolina State University, 5109 SAS Hall, Campus Box 8203, Raleigh, North Carolina 27695, USA \email{sghosal@ncsu.edu} }
\maketitle

\begin{abstract}
We show that rate-adaptive multivariate density estimation can be performed using Bayesian methods based on Dirichlet mixtures of normal
kernels with a prior distribution on the kernel's covariance matrix parameter.
We derive sufficient conditions on the prior specification that guarantee convergence to
a true density at a rate that is optimal minimax for the smoothness class to which the true density belongs.
No prior knowledge of smoothness is assumed.
The sufficient conditions are shown to hold for the Dirichlet location mixture of normals prior with a Gaussian base measure and
an inverse-Wishart prior on the covariance matrix parameter. Locally H\"older smoothness classes and their anisotropic extensions are considered.
Our study involves several technical novelties,
including sharp approximation of finitely differentiable multivariate densities by normal mixtures and a new sieve on the space of such densities.
\end{abstract}


\section {Introduction}
\label{sec1}
Asymptotic frequentist properties of Bayesian non-parametric methods have recently received much attention. It is now recognized that a single fully Bayesian method can offer adaptive optimal rates of convergence for large collections of true data generating distributions ranging over several smoothness classes. Examples include signal estimation in the presence of Gaussian white noise \citep{Belitser2003}, density estimation and regression based on a mixture model of spline or wavelet bases \citep{Huang2004, Ghosal2008}, regression, classification and density estimation based on a rescaled Gaussian process model \citep{Vander2009}, density estimation based on a hierarchical finite mixture model of beta densities \citep{Rousseau2010}, density estimation \citep{Kruijer2010} and regression \citep{Jonge2010} based on hierarchical, finite mixture models of location-scale kernels.

Adaptive convergence rates results for nonparametric Bayesian methods are useful for at least two reasons.
First, they provide frequentist justification of these methods in large samples,
which can be attractive to non-Bayesian practitioners who use these methods because they are easy to implement,
provide estimation and prediction intervals, do not require adjusting tuning parameters and can handle multivariate data.
Second, these results are an indirect validation that the spread of the underlying prior distribution is well balanced across its infinite dimensional support.
Such a prior distribution quantifies the rate at which it packs mass into a sequence of shrinking neighborhoods around any given point in its support.
When the support of the prior can be partitioned into smoothness classes in the space of continuous functions,
a sharp bound for this rate can be calculated for all support points within each smoothness class.
These calculations have a nearly one-to-one relationship with the asymptotic convergence rates of the resulting method.

In this article we focus on a collection of nonparametric Bayesian density estimation methods based on Dirichlet process mixture of normals priors.
%In this paper we show that a collection of multivariate density estimation methods based on widely used Dirichlet process mixture of normals priors are rate adaptive across all H\"older smoothness classes of density functions.
Dirichlet process mixture priors \citep{Ferguson19832, Lo1984} form a cornerstone of nonparametric Bayesian methodology
\citep{Escobar1995, Muller1996, Muller2004, Dunson20101} and density estimation methods based on these priors
are among the first Bayesian nonparametric methods for which convergence results were obtained \citep{Ghosal1999, Ghosal20011, Tokdar2006}.
However, due to two major technical difficulties, rate adaptation results have not yet been available
and convergence rates remain unknown beyond univariate density estimation \citep{Ghosal20011,Ghosal2007}.
The first major difficulty lies in showing adaptive prior concentration rates for mixture priors on density functions.
Taylor expansions do not suffice because of the non-negativity constraint on the densities.
The second major difficulty is to construct a suitable low-entropy, high-mass sieve on the space of infinite component mixture densities.
Such sieve constructions are an integral part of current technical machinery for deriving rates of convergence.
The sieves that have been used to study Dirichlet process mixture models \citep[e.g., in][]{Ghosal2007} do not scale to higher
dimensions and lack the ability to adapt to smoothness classes \citep{Wu2010}.

We plug these two gaps and establish rate adaptation properties of a collection of multivariate density estimation methods based on
Dirichlet process mixture of normals priors.
Our priors include the commonly used specification of mixing over multivariate normal kernels with a location parameter drawn from a Dirichlet process with a Gaussian base measure while using an inverse-Wishart prior on the common covariance matrix parameter of the kernels. Rate adaptation is established with respect to H\"older smoothness classes. In particular, when any density estimation method from our collection is applied to independent observations $X_1, \ldots, X_n \in \mathbb{R}^d$ drawn from a density $f_0$ which belongs to the smoothness class of locally $\beta$-H\"older functions, it is shown to produce a posterior distribution on the unknown density of $X_i$'s that converges to $f_0$ at a rate $n^{-\beta / (2\beta + d)}(\log n)^t$, where $t$ depends on $\beta$, $d$ and tail properties of $f_0$. This rate, without the $(\log n)^t$ term, is minimax optimal for the $\beta$-H\"older class \citep{Barron1999}. It is further shown that if $f_0$ is anisotropic with H\"older smoothness coefficients $\beta_1, \ldots, \beta_d$ along the $d$ axes, then the posterior convergence rate is $n^{-\beta_0 / (2\beta_0 + d)}$ times a $\log n$ factor, where $\beta_0$ is the harmonic mean of $\beta_1, \ldots, \beta_d$. Again this rate is minimax optimal for this class of functions \citep{Hoffmann2002}.

To the best of our knowledge, such rate adaptation results are new for any kernel based multivariate density estimation method.
The performance of a non-Bayesian, multivariate kernel density estimator depends heavily on the difficult
choice of a bandwidth and a smoothing kernel \citep{Scott1992}. Optimal rates are possible only by using higher order kernels and
the choices of bandwidth that require knowing the smoothness level. In contrast our results show that a single Bayesian nonparametric method based on a single choice of Dirichlet process mixture of normal kernels achieves optimal convergence rates universally across all smoothness levels.

\section{Posterior Convergence Rates for Dirichlet Mixtures}
\label{sec2}
\subsection{Notation}
\label{subsec2.1}

For any $d\times d$ positive definite real matrix $\Sigma$, let $\phi_{\Sigma}(x)$ denote the $d$-variate normal density $(2\pi)^{-d/2}(\det \Sigma)^{-1/2}\exp(-x^T\Sigma^{-1}x / 2)$ with mean zero and covariance matrix $\Sigma$. For a probability measure $F$ on $\mathbb{R}^d$ and a $d\times d$ positive definite real matrix $\Sigma$, the $F$ induced location mixture of $\phi_\Sigma$ is denoted by $p_{F,\Sigma}$, i.e., $p_{F, \Sigma}(x) = \int \phi_\Sigma(x - z) F(dz)$ $(x \in \mathbb{R}^d)$. For a scalar $\sigma > 0$ and any function $f$ on $\mathbb{R}^d$, we let $K_\sigma f$ denote the convolution of $f$ and $\phi_{\sigma^2 I}$, i.e., $(K_\sigma f)(x) = \int \phi_{\sigma^2 I}(x - z) f(z)dz$.

For any finite positive measure $\alpha$ on $\mathbb{R}^d$, let $\mathcal{D}_\alpha$ denote the Dirichlet process distribution with parameter $\alpha$ \citep{Ferguson1973}. That is, an $F \sim \mathcal{D}_\alpha$ is a random probability measure on $\mathbb{R}^d$, such that for any Borel measurable partition $B_1, \ldots, B_k$ of $\mathbb{R}^d$ the joint distribution of $F(B_1), \ldots, F(B_k)$ is the $k$-variate Dirichlet distribution with parameters $\alpha(B_1), \ldots, \alpha(B_k)$.

Let $\natz = \{0,1,2,\ldots\}$ and let $\Delta_J = \{(x_1,\ldots,x_J): x_i > 0, i=1,\ldots,J; \sum_{i=1}^J x_i =1\}$ denote the $J$-dimensional probability simplex. Let the indicator function of a set $A$ be denoted by $\Ind(A)$.
We use $\lesssim$ to denote an inequality up to a constant multiple, where the underlying
constant of proportionality is universal or is unimportant for our purposes.
For any $x \in \mathbb{R}$, define $\lfloor x\rfloor$ as the largest integer that is strictly smaller than $x$.
Similarly, define $\lceil x \rceil $ as the smallest integer strictly greater than $x$.
For a multi-index $k = (k_1, \ldots, k_d) \in \natz^d$, define $k_\cdot = k_1 + \cdots + k_d$, $k! = k_1 ! \cdots k_d!$
and let $D^k$ denote the mixed partial derivative operator $\partial^{k_\cdot} / \partial x_1^{k_1}\cdots \partial x_d^{k_d}$.

For any $\beta > 0$, $\tau_0 \geq 0$ and any non-negative function $L$ on $\mathbb{R}^d$, define the locally $\beta$-H\"older class with envelope $L$, denoted $\mathcal{C}^{\beta, L, \tau_0}(\mathbb{R}^d)$, to be the set of all functions $f: \mathbb{R}^d \to \mathbb{R}$ with finite mixed partial derivatives $D^k f$  $(k \in \natz^d)$ of all orders up to $k_\cdot \le \lfloor \beta \rfloor$, and for every $k \in \natz^d$ with $k_\cdot = \lfloor\beta\rfloor$ satisfying
\begin{equation}\label{smoothness}
|(D^kf)(x + y) - (D^kf)(x)| \le L(x)e^{\tau_0 \|y\|^2} \|y \|^{\beta - \lfloor\beta\rfloor},~~~~x, y \in \mathbb{R}^d.
\end{equation}
In our discussion, we shall assume that the true density $f$ lies in $\mathcal{C}^{\beta, L, \tau_0}(\mathbb{R}^d)$.
This condition is essentially weaker than the one in \citet{Kruijer2010},
where $\log f \in \mathcal{C}^{\beta, L, 0}(\mathbb{R})$ is assumed, see Lemma \ref{rm1}.

For any $d\times d$ matrix $A$, we denote its eigenvalues by $\eig_1(A) \leq \cdots \leq \eig_d(A)$, its spectral norm by $\|A\|_2 = \sup_{x \ne 0} \|Ax\| / \|x\|$ and its max norm by $\|A\|_{\max}$, the maximum of the absolute values of the elements of $A$.

\subsection {Dirichlet process mixture of normals prior}
\label{sec: prior}

Consider drawing inference on an unknown probability density function $f$ on $\mathbb{R}^d$ based on independent observations $X_1, \ldots, X_n$ from $f$. A nonparametric Bayesian method assigns a prior distribution $\Pi$ on $f$ and draws inference on $f$ based on the posterior distribution $\Pi_n(\cdot \mid X_1, \ldots, X_n)$. A Dirichlet process location mixture of normals prior $\Pi$ is the distribution of a random probability density function $p_{F,\Sigma}$, where $F \sim \mathcal{D}_\alpha$ for some finite positive measure $\alpha$ on $\mathbb{R}^d$ and $\Sigma \sim G$, a probability distribution on $d\times d$ positive definite real matrices.

We restrict our discussion to a collection of such prior distributions $\Pi$ for which the associated $\mathcal{D}_\alpha$ and $G$ satisfy the following conditions.
Let $|\alpha| = \alpha(\mathbb{R}^d)$ and $\bar{\alpha}= \alpha/|\alpha|$. We assume that $\bar \alpha$
has a positive density function on the whole of $\mathbb{R}^d$ and that there exist positive constants $a_1, a_2, a_3, b_1, b_2, b_3, C_1, C_2$
such that
\begin{align}
1 - \bar{\alpha} ([-x,x]^d) & \leq b_1\exp(-C_1 x^{a_1})~~\mbox{for all sufficiently large}~x > 0,\label{eq:dp1}\\
G\{\Sigma: \eig_d (\Sigma^{-1}) \geq x\} & \leq b_2\exp (-C_2 x^{a_2})~~\mbox{for all sufficiently large}~x>0, \label{eq:iw1}\\
G\{\Sigma: \eig_1 (\Sigma^{-1}) < x\} & \leq b_3 x^{a_3}~~\mbox{for all sufficiently small}~x > 0,  \label{eq:iw2}
 \end{align}
and that there exist $\kappa, a_4, a_5, b_4, C_3>0$ such that for any $0 < s_1 \le \cdots \le s_d$ and $t \in (0, 1)$,
%diagonal matrix $A = \diag \{\lambda_1,\ldots,\lambda_d\}$ with each $\lambda_i \in (0, \epsilon)$ and any $x \in (0, \epsilon)$
\begin{equation}
 G\{\Sigma: s_j < \eig_j(\Sigma^{-1}) < s_j(1 + t), j = 1,\ldots,d\}
 %(|\Tr(A\Sigma^{-1}) - d -\log \det (A \Sigma^{-1}) |\leq x) \gtrsim \left(\prod_{i=1}^d \lambda_i^{a_4}\right) \exp\left(-C_3 \sum_{i=1}^d \lambda_i^{-1}\right) x^{a_5}.
 \ge b_4 s_1^{a_4}t^{a_5}\exp(-C_3 s_d^{\kappa/2}).
  \label{eq:iw3}
  \end{equation}

Our assumption on $\bar{\alpha}$ is analogous to (11) of \citet{Kruijer2010} and holds, for example, when $\bar \alpha$ is a Gaussian measure on $\mathbb{R}^d$. Unlike previous treatments of Dirichlet process mixture models \citep{Ghosal20011,Ghosal2007}, we allow a full support prior on $\Sigma$ including the widely used inverse-Wishart distribution. The following lemma shows that such a $G$ satisfies our assumptions; see Appendix A for a proof.

\begin{lemma}
\label{lem:inw}
The inverse-Wishart distribution $\text{IW}(\nu,\Psi)$ with $\nu$ degrees of freedom and a positive definite scale matrix $\Psi$ satisfies \eqref{eq:iw1}, \eqref{eq:iw2} and \eqref{eq:iw3} with $\kappa = 2$.
\end{lemma}

From a computational point of view, another useful specification is to consider a $G$ that
supports only diagonal covariance matrices $\Sigma = \diag(\sigma_1^2, \ldots, \sigma_d^2)$,
with each diagonal component independently assigned a prior distribution $G_0$.
By choosing an inverse gamma distribution as $G_0$, we get a $G$ that again satisfies \eqref{eq:iw1}, \eqref{eq:iw2} and \eqref{eq:iw3} with $\kappa = 2$.
Alternatively, we could take $G_0$ to be the distribution of the square of an inverse gamma random variable.
Such a $G_0$ leads to a $G$ that satisfies \eqref{eq:iw1}, \eqref{eq:iw2} and \eqref{eq:iw3} with $\kappa = 1$.
This difference in $\kappa$ matters, with smaller $\kappa$ leading to optimal convergence rates for a wider class of true densities.

%\section {Main results}
% Setting model numbering
%\mathbb{R}newcommand{\theequation}{M.\arabic{equation}}

%Formally, we assume that the prior distribution $\Pi$ on the space of probability densities on $\mathbb{R}^d$ is defined with respect to the Borel $\sigma$-field on this space induced by a metric $\rho$. We will restrict to two choices of the metric, the $L_1$ metric $\rho(f, g) = \|f - g\|_1 := \int | f(x) -  g(x)|dx$ and the Hellinger metric $\rho(f, g) = h(f, g) := [\int \{f^{1/2}(x) - g^{1/2}(x)\}^2 dx]^{1/2}$. These two metrics induce the same Borel $\sigma$-field because of the well known relation $(1/2)\|f - g\|_1 \le h(f, g) \le \|f - g\|_1^{1/2}$.

\subsection{Convergence rates results}
Let $\Pi$ be a Dirichlet process mixture prior as defined in Section \ref{sec: prior} and let $\Pi_n(\cdot \mid X_1, \ldots, X_n)$ denote the posterior distribution based on $n$ observations $X_1, \ldots, X_n$ modeled as $X_i \sim f$, $f \sim \Pi$. Let $\left\{\epsilon_n\right\}_{n \ge 1}$ be a sequence of positive numbers with $\lim_{n \to \infty} \epsilon_n = 0$. Also let $\rho$ denote a suitable metric on the space of probability densities on $\mathbb{R}^d$, such as the $L_1$ metric $\|f - g\|_1 = \int |f(x) - g(x)|dx$ or the Hellinger metric $d_H(f, g) = [\int \{f^{1/2}(x) - g^{1/2}(x)\}^2 dx]^{1/2}$. Fix any probability density $f_0$ on $\mathbb{R}^d$. For the density estimation method based on $\Pi$ we say its posterior convergence rate at $f_0$ in the metric $\rho$ is $\epsilon_n$ if for any $M < \infty$
\begin{eqnarray}
\label{def rate}
\lim_{n \to 0} \Pi_n\left[\{f: \rho(f_0, f) > M \epsilon_n\} | X_1, \ldots, X_n\right] = 0~~\mbox{almost surely},
\end{eqnarray}
whenever $X_1, X_2,\ldots$ are independent and identically distributed with density $f_0$.

Although \eqref{def rate} only establishes $(\epsilon_n)_{n \ge 1}$ as a bound on the convergence rate at $f_0$, it serves as a useful calibration when checked against the optimal rate for the smoothness class to which $f_0$ belongs. It is known that the minimax rate associated with a $\beta$-H\"older class is $n^{-\beta / (2\beta + d)}$. We establish \eqref{def rate} for this class with $\epsilon_n$ as $n^{-\beta / (2\beta + d)}$ up to a factor a power of $\log n$. A formal result requires some additional conditions on $f_0$, as summarized in the theorem below.

\begin{theorem}\label{thm:x2}
Suppose that $f_0 \in \mathcal{C}^{\beta, L, \tau_0}(\mathbb{R}^d)$ is a probability density function satisfying
\begin{equation}
P_0\left(|D^kf_0|/f_0\right)^{{(2\beta + \epsilon)}/{k_\cdot}} < \infty,~~k\in \natz^d, k_\cdot \le \lfloor\beta\rfloor,~~~~P_0\left( L/f_0\right)^{{(2\beta + \epsilon)}/{\beta}}  < \infty
\label{eq:int}
\end{equation}
for some $\epsilon > 0$ where $P_0g = \int g(x)f(x)dx$ denotes expectation of $g(X)$ under $X \sim f_0$. Also suppose there are positive constants $a, b, c, \tau$ such that
\begin{equation}
f_0(x) \le c \exp(-b\|x\|^\tau),~~\|x\| > a.
\label{eq:tail}
\end{equation}
 For the prior $\Pi$ constructed in Section \ref{sec: prior}, \eqref{def rate} holds in the Hellinger or the $L_1$ metric with $\epsilon_n = n^{-\beta / (2\beta + \dstar)} (\log n)^t$, where $t > \{ \dstar (1 + 1/\tau + 1 / \beta) + 1\} /(2 + \dstar/\beta)$ and $\dstar = \max(d, \kappa)$.
\end{theorem}

We prove this result by verifying a set of sufficient conditions presented originally in \citet{Ghosal2000} and subsequently modified by \citet{Ghosal2007}. For any subset $A$ of a metric space equipped with a metric $\rho$ and an $\epsilon > 0$, let $N(\epsilon, A, \rho)$ denote the $\epsilon$-covering number of $A$, i.e., $N(\epsilon, A, \rho)$ is the smallest number of balls of radius $\epsilon$ needed to cover $A$. The logarithm of this number is referred to the $\epsilon$-entropy of $A$. Also define $\mathcal{K}(f_0, \epsilon)= \{f: \int f_0 \log (f_0/f) < \epsilon^2, \, \, \int f_0 \log^2 (f_0/f) < \epsilon^2 \, \}$ as the Kullback--Leibler ball around $f_0$ of size $\epsilon$. \citet{Ghosal2007} show that \eqref{def rate} holds whenever there exist positive constants $c_1, c_2, c_3, c_4$, a sequence of positive numbers $(\tilde \epsilon_n)_{n \ge 1}$ with $\tilde \epsilon_n \le \epsilon_n$ and $\lim_{n\to \infty} n\tilde\epsilon_n^2 = \infty$ and a sequence of compact subsets $(\mathcal{F}_n)_{n \ge 1}$ of probability densities satisfying,
\begin{eqnarray}
& \log N(\epsilon_n, \mathcal{F}_{n},\rho) \leq c_1 n\epsilon_n^2, \label{eq:e28} \\
& \Pi ( \mathcal{F}_n ^c ) \leq c_3 e^{-(c_2 + 4)n \tilde \epsilon_n^2 }, \label{eq:e30} \\
& \Pi\left\{ \mathcal{K}(f_0 , \tilde{\epsilon}_n ) \right\} \geq c_4 e^{- c_2 n \tilde{\epsilon}_n ^2 }. \label{eq:e29}
\end{eqnarray}
The sequence of sets $\mathcal{F}_n$ is often called a sieve and the Kullback--Leibler ball probability in \eqref{eq:e29} is called the prior thickness at $f_0$. In Theorem \ref{thm:thick} we show that \eqref{eq:e29} holds for $\Pi = \mathcal{D}_\alpha \times G$ with $\tilde \epsilon_n = n^{-\beta / (2\beta + \dstar)}(\log n)^{t_0}$ where $t_0 = \{\dstar(1 + 1 / \tau + 1 / \beta) + 1\} / (2 + \dstar / \beta)$. In Theorem \ref{thm:sieve} we show that \eqref{eq:e28}, \eqref{eq:e30} hold with $\tilde\epsilon_n$ as before and $\epsilon_n = n^{-\beta / (2\beta + \dstar)}(\log n)^t$ for every $t > t_0$. The following sections lay out the machinery needed to establish these two fundamental results.

When $\kappa = 1$, the rate in Theorem \ref{thm:x2} equals the optimal rate $n^{-\beta / (2\beta + d)}$ up to a factor $\log n$. However, the commonly used inverse Wishart specification of $G$ leads to $\kappa = 2$, and hence Theorem \ref{thm:x2} gives the optimal rate only for $d \ge 2$. We will later see that $\kappa$ has a bigger impact on rates of convergence for anisotropic densities.

Our result also applies for a finite mixture prior specification $\Pi$, where the density function $f$ is
represented by $f(x) = \sum_{h=1}^H \omega_h \phi_{\Sigma} (x-\mu_h)$
and priors are assigned on $H$, $\Sigma$, $\omega = (\omega_1,\ldots,\omega_H)$ and $\mu_1,\ldots,\mu_H$.
We assume $\Sigma \sim G$, which satisfies \eqref{eq:iw1}, \eqref{eq:iw2} and \eqref{eq:iw3}, and that there exist positive constants $a_4$, $b_4$, $b_5$, $b_6$, $b_7$, $C_4$, $C_5$, $C_6$, $C_7$ such that
$b_4 \exp\{ -C_4 x (\log x)^{\tau_1}\} \leq \Pi (H \geq x)   \leq b_5 \exp\{ -C_5 x (\log x)^{\tau_1}\}$ for sufficiently large $x>0$,
while for every fixed $H=h$,
\begin{align}
  \Pi  (\mu_i \notin [-x,x]^d) & \leq b_6\exp(-C_6 x^{a_4}),~~\mbox{for sufficiently large}~x>0, ~~i=1,\ldots,h,  \nn\\
\Pi (\|\omega - \omega_0\| \leq \epsilon) & \geq   b_7 \exp \{-C_7 h \log (1/\epsilon)\}, ~~\mbox{for all}~0< \epsilon < 1/h ~\mbox{and all}~\omega_0 \in \Delta_h.  \nn
\end{align}
Theorem \ref{thm:ratefm} summarizes our findings for a finite mixture prior. Its proof is similar to that of Theorem \ref{thm:x2}
except that in verifying \eqref{eq:e30}, we need $\exp \{- H (\log H)^{\tau_1}\} \lesssim \exp\{-n \tilde{\epsilon}_n^2\}$. Together with $H=\lfloor n\epsilon_n^2/(\log n)\rfloor$, we have $ \epsilon_n^2 (\log n)^{\tau_1 - 1} \geq \tilde{\epsilon}_n^2 $, leading to  $\tilde \epsilon_n = n^{-\beta / (2\beta + \dstar)}(\log n)^{t_0}$ where $t_0 = \{\dstar(1 + 1 / \tau + 1 / \beta) + 1\} / (2 + \dstar / \beta)$ and $\epsilon_n = n^{-\beta / (2\beta + \dstar)}(\log n)^t$ with $t > t_0 + \max\left\{0,(1-\tau_1)/2\right\}$.
\begin{theorem}\label{thm:ratefm}
Suppose that $f_0 \in \mathcal{C}^{\beta, L, \tau_0}(\mathbb{R}^d)$ is a probability density function satisfying \eqref{eq:int} and \eqref{eq:tail} for some positive constants $a, b, c, \tau, \epsilon$. For a finite mixture prior $\Pi$ as in above, \eqref{def rate} holds in the Hellinger or the $L_1$ metric with $\epsilon_n = n^{-\beta / (2\beta + \dstar)} (\log n)^t$ for every $t > \{ \dstar (1 + 1/\tau + 1 / \beta) + 1\} /(2 + \dstar/\beta) + \max\left\{0,(1-\tau_1)/2\right\}$, where $\dstar = \max(d, \kappa)$.
\end{theorem}

\section{Prior thickness results}
\label{sec3}

Functions in $\mathcal{C}^{\beta, L, \tau_0}$ can be approximated by mixtures of $\phi_{\sigma^2 I}$ with an accuracy that improves with $\beta$.
We establish this through the following constructions and lemma, which are adapted from Lemma 3.4 of \citet{Jonge2010}
and univariate approximation results of \citet{Kruijer2010}. The proofs are given in Appendix A.

For each $k \in \natz^d$, let $m_k$ denote the $k$-th moment $m_k = \int y^k \phi_1(y)dy$ of the standard normal distribution on $\mathbb{R}^d$.
For $n \in \natz^d$, define two sequences of numbers through the following recursion.
If $n_\cdot = 1$, set $c_n = 0$ and $d_n = -m_n / n!$, and for $n_\cdot \ge 2$ define
\begin{equation}
c_n = -\sum_{\substack{n\,=\,l + k\\ l_\cdot \ge 1,~ k_\cdot \ge 1}} \frac{(-1)^{k_{\cdot}}}{k!}m_kd_l,~~~~~~ d_n = \frac{(-1)^{n_\cdot}m_n}{n!} + c_n.
\label{eq:cd}
\end{equation}
 Given $\beta > 0$, $\sigma > 0$, define a transform $T_{\beta, \sigma}$ on
$f:\mathbb{R}^d \to \mathbb{R}$ with derivatives up to order $\lfloor \beta \rfloor$ by
\begin{equation}
T_{\beta,\sigma}f = f - \sum_{\substack{k\,\in\,\natz^d\\ 1 \le k_\cdot \le \lfloor\beta\rfloor}} d_k \sigma^{k_\cdot} D^kf.
\label{ref:T}
\end{equation}

\begin{lemma}
\label{lem:basic}
For any $\beta, \tau_0 > 0$ there is a positive constant $M_\beta$ such that any $f \in \mathcal{C}^{\beta, L, \tau_0}(\mathbb{R}^d)$, it satisfies
$
|\{K_\sigma (T_{\beta,\sigma}f) - f\}(x)| < M_\beta L(x)\sigma^\beta
$
for all $x \in \mathbb{R}^d$ and all $\sigma \in (0, 1/(2\tau_0)^{1/2})$.
\end{lemma}

Lemma \ref{lem:basic} applies to any functions $f \in \mathcal{C}^{\beta, L, \tau_0}$, not necessarily a probability density,
and the mixing function $T_{\beta, \sigma}f$ need not be a density and could be negative.
Fortunately, when $f$ is a probability density, we can derive a density $h_{\sigma}$ from $T_{\beta, \sigma}f$
so that $K_\sigma h_{\sigma}$ provides a $\sigma^{\beta}$ order approximation to $f$. The construction of $h_{\sigma}$ can be viewed as a multivariate extension of results in Section 3 of \citet{Kruijer2010}.
The main difference is that we establish approximation results under the Hellinger distance and apply Taylor expansions on $f_0$ instead of $\log f_0$, which lead to a more elegant proof.

\begin{theorem}
\label{thm:hell approx}
Let $f_0 \in \mathcal{C}^{\beta, L, \tau_0}(\mathbb{R}^d)$ be a probability density function and write $f_{\sigma} = T_{\beta, \sigma}f_0$. Suppose that $f_0$ satisfies \eqref{eq:int} for some $\epsilon > 0$.  Then there exist $s_0 > 0, K > 0$ such that for any $0 < \sigma < s_0$,
$
g_{\sigma} = f_{\sigma} + \frac12f_0 \,\Ind\{f_{\sigma} < (1/2)f_0\}
$
is a non-negative function with $\int g_{\sigma}(x)dx < \infty$ and the density  $h_{\sigma} = g_{\sigma} / \int g_{\sigma}(x)dx$ satisfies $
d_H^2(f_0, K_\sigma h_\sigma) \le K \sigma^{2\beta}.$
\end{theorem}

The next result trades $g_\sigma$ for a compactly supported density $h_\sigma$ whose convolution with $\phi_{\sigma^2 I}$ inherits the same order $\sigma^\beta$ approximation to $f_0$. We need the tail condition \eqref{eq:tail} on $f_0$ to obtain a suitable compact support.

\begin{proposition}
\label{prop:compact}
Let $f_0 \in \mathcal{C}^{\beta,L,\tau_0}(\mathbb{R}^d)$ be a probability density function satisfying \eqref{eq:int}
and \eqref{eq:tail} for some positive constants  $\epsilon, a, b, c, \tau$. For any $\sigma > 0$, define $E_\sigma = \{x \in \mathbb{R}^d: f_0(x) \ge \sigma^{(4\beta + 2\epsilon + 8)/\delta}\}$. Then there exist $s_0, a_0, B_0, K_0 > 0$ such that for every $0 < \sigma < s_0$, $P_0(E_\sigma^c) \le B_0\sigma^{4\beta + 2\epsilon + 8}$, $E_\sigma \subset \{x\in \mathbb{R}^d:\|x\|\le a_\sigma\}$ where $a_\sigma = a_0 \{\log (1/\sigma)\}^\tau$ and there is a probability density $\tilde h_\sigma$ with support inside $\{x\in \mathbb{R}^d: \|x\| \le  a_\sigma\}$ satisfying $d_H(f_0, K_\sigma \tilde h_\sigma) \le K_0 \sigma^\beta$.
\end{proposition}

Proposition \ref{prop:compact} paves the way to calculating prior thickness around $f_0$ because the probability density $K_\sigma \tilde h_\sigma$
can be well approximated by densities $p_{F,\Sigma}$ with $(F, \Sigma)$ chosen from a suitable set.
Toward this we present the final theorem of this section and a proof of it that overlaps with Section 9 of \citet{Ghosal2007}.
However, our proof requires new calculations to handle a non-compactly supported $f_0$ and a matrix valued $\Sigma$.

\begin{theorem}
\label{thm:thick}
Let $f_0 \in \mathcal{C}^{\beta,L, \tau_0}(\mathbb{R}^d)$ be a bounded probability density function satisfying \eqref{eq:int} and \eqref{eq:tail} for some positive constants $\epsilon, a, b, c, \tau$. Then for some $A, C > 0$ and all sufficiently large $n$,
\begin{equation}\label{eq:inthick}
(\mathcal{D}_\alpha \times G)\left\{(F, \Sigma) : P_0\log\frac{f_0 }{p_{F,\Sigma}} \le A\tilde \epsilon_n^2,~~P_0\left(\log\frac{f_0 }{ p_{F,\Sigma}}\right)^2 \le A\tilde \epsilon_n^2\right\} \ge e^{-Cn\tilde\epsilon_n^2}
\end{equation}
 where $\tilde \epsilon_n = n^{-\beta / (2\beta + \dstar)}(\log n)^t$ with any $t \ge \{\dstar(1 + 1 / \tau + 1 / \beta) + 1\} / (2 + \dstar/\beta)$.
\end{theorem}

\begin{proof}
Let $\delta, s_0, a_0, K_0$ be as in Proposition \ref{prop:compact}. Consider $n$ large enough so that $\tilde \epsilon_n  < s_0^{\beta}$. Fix $\sigma^{\beta} = \tilde \epsilon_n \{\log (1/\tilde\epsilon_n)\}^{-1}$ and as in Proposition \ref{prop:compact} define $E_\sigma = \{x \in \mathbb{R}^d: f_0(x) \ge \sigma^{(4\beta + 2\epsilon + 8)/\delta}\}$ and $a_\sigma = a_0 \{\log(1/\sigma)\}^{1/\tau}$. Recall that $P_0(E_\sigma^c) \le B_0\sigma^{4\beta + 2\epsilon + 8}$ for some constant $B_0$ and $E_\sigma \subset \{x\in \mathbb{R}^d:\|x\|\le a_\sigma\}$.  Apply Proposition \ref{prop:compact} to find $\tilde h_\sigma$ with support $E_\sigma$ such that $d_H(f_0, K_\sigma \tilde h_\sigma) \le K_0\sigma^\beta$. Find $b_1 > \max(1, 1 / 2\beta)$ such that $\tilde\epsilon_n^{b_1}\{\log(1/\tilde\epsilon_n)\}^{5/4} \le \tilde \epsilon_n$.

By Corollary \ref{cor:disc approx} there is a discrete probability measure $F_\sigma = \sum_{j = 1}^N p_j\delta_{z_j}$, with at most
$N \le D_0 \sigma^{-d}\{\log(1/\sigma)\}^{d/\tau}\{\log (1/\tilde \epsilon_n)\}^{d} \le D_1 \sigma^{-d}\{\log (1/\tilde\epsilon_n)\}^{d + d / \tau}
$ many support points inside $\{x\in \mathbb{R}^d:\|x\|\le a_\sigma\}$, with at least $\sigma\tilde \epsilon_n^{2b_1}$ separation between any $z_i \ne z_j$ such that $d_H(K_\sigma \tilde h_\sigma, K_\sigma F_\sigma) \le A_1 \tilde\epsilon_n^{b_1}\{\log(1/\tilde\epsilon_n)\}^{1/4}$ for some constants $A_1, D_1$.

Place disjoint balls $U_j$ with centers at $z_1,\ldots,z_N$ with diameter $\sigma \tilde\epsilon_n^{2b_1}$ each.
Extend $\{U_1, \ldots, U_N\}$ to a partition $\{U_1, \ldots, U_K\}$ of $\{x \in \mathbb{R}^d: \|x\| \le a_\sigma\}$
such that each $U_j$ $(j = N + 1, \ldots, K)$ has a diameter at most $\sigma$. This can be done with $K \le D_2 \sigma^{-d}\{\log(1/\tilde\epsilon_n)\}^{d + d / \tau}$ for some constant $D_2$. Further extend this to a partition $U_1, \ldots, U_M$ of $\mathbb{R}^d$ such that $a_1(\sigma\tilde\epsilon_n^{2b_1})^d \le \alpha(U_j) \le 1$ for all $j = 1, \ldots, M$ for some constant $a_1$. We can still have $M \le D_3 \sigma^{-d}\{\log(1/\tilde\epsilon_n)\}^{d + d/\tau} \le D_4 \tilde\epsilon_n^{-d/\beta}\{\log(1/\tilde\epsilon_n)\}^{sd}$ with $s = 1 + 1/\beta + 1 / \tau$, for some constants $D_3, D_4$. None of these constants depends on $n$ or $\sigma$.

Define $p_j = 0$  $(j = N + 1, \ldots, M)$. Let $\mathcal{P}_\sigma$ denote the set of probability measures $F$ on $\mathbb{R}^d$ with $\sum_{j = 1}^M |F(U_j) - p_j| \le 2\tilde\epsilon_n^{2db_1}$ and $\min_{1 \le j \le M} F(U_j) \ge \tilde\epsilon_n^{4db_1}/2$. Note that
\begin{align*}
M\tilde\epsilon_n^{2db_1} & \le D_4 [\tilde\epsilon_n^{b_1 - 1/(2\beta)}\{\log(1/\tilde\epsilon_n)\}^{s/2}]^{2d} \le 1, \\
\min_{1\le j\le M}\alpha(U_j)^{1/2}& \ge a_1^{1/2} \tilde\epsilon_n^{2db_1} \{\tilde\epsilon_n^{b_1 - 1/(2\beta)}\log(1/\tilde\epsilon_n)\}^{-d} \ge (a_1 / D_4)^{1/2} \tilde\epsilon_n^{2db_1},
\end{align*}
provided $n$ has been chosen large enough. By Lemma 10 of \cite{Ghosal2007}, $\mathcal{D}_\alpha(\mathcal{P}_\sigma) \ge C_1 \exp\{-c_1M\log(1/\tilde\epsilon_n)\} \ge C_1\exp[-c_2 \tilde\epsilon_n^{-d/\beta}\{\log(1/\tilde\epsilon_n)\}^{sd + 1}]$ for some constants $C_1, c_2$ that depend on $\alpha(\mathbb{R}^d)$, $a_1$, $D_4$, $d$ and $b_1$. Also, let $\mathcal{S}_\sigma$ denote the set of all $d\times d$ non-singular matrices $\Sigma$ such that all eigenvalues of $\Sigma^{-1}$ lie between $\sigma^{-2}$ and $\sigma^{-2}(1 + \sigma^{2\beta})$. By \eqref{eq:iw3}, $G(\mathcal{S}_\sigma) \ge \sigma^{D_5}\exp(-D_6/\sigma^{\kappa}) \ge C_3\exp[-c_3 \tilde\epsilon_n^{-\kappa/\beta}\{\log(1/\tilde\epsilon_n)\}^{s\kappa + 1}]$ for some constants $C_3, c_3$. Any $\Sigma \in \mathcal{S}_\sigma$ satisfies $\det(\Sigma^{-1}) \ge \sigma^{-2d}$, $y^T\Sigma^{-1}y \le 2\|y\|^2 / \sigma^2$ for any $y \in \mathbb{R}^d$ and $|\mathrm{tr}(\sigma^2\Sigma^{-1}) - d - \log\det(\sigma^2\Sigma^{-1})| < d\sigma^{2\beta}$.

Apply Lemma \ref{lem:part approx} with $V_i = U_i$  $(i = 1, \ldots, N)$ and $V_0 = \cup_{j > N} U_j$ to conclude that for any $F \in \mathcal{P}_\sigma$, $d_H(K_\sigma F_\sigma, K_\sigma F) \le A_2 \tilde\epsilon_n^{b_1}$ for some universal constant $A_2$ and hence
\begin{align*}
d_H(f_0, K_\sigma F) & \le d_H(f_0, K_\sigma \tilde h_\sigma) + d_H(K_\sigma \tilde h_\sigma, K_\sigma F_\sigma) + d_H(K_\sigma F_\sigma, \phi_{\sigma^2 I}  * F)\\
& \le K_0\sigma^{\beta} + A_1 \tilde\epsilon_n^{b_1}\{\log(1/\tilde\epsilon_n)\}^{1/4} + A_2\tilde\epsilon_n^{b_1}  \le A_3 \sigma^\beta
\end{align*}
for some constant $A_3$. Therefore, for any $F \in \mathcal{P}_\sigma$, $\Sigma \in \mathcal{S}_\sigma$,
$
d_H(f_0, p_{F,\Sigma}) \le d_H(f_0, K_\sigma F) + d_H(p_{F, \sigma^2I}, p_{F, \Sigma}) \le A_4\sigma^\beta
$
for some constant $A_4$ because $d_H(p_{F, \sigma^2I}, p_{F, \Sigma}) \le |\mathrm{tr}(\sigma^2\Sigma^{-1}) - d - \log\det(\sigma^2\Sigma^{-1})|^{1/2}$ for any $F$. Moreover, for every $x \in \mathbb{R}^d$ with $\|x\| < a_\sigma$,
\[
\frac{p_{F,\Sigma}(x)}{f_0(x)} \ge \frac{K_1}{\sigma^d} \int_{\|x - z\|\le \sigma} \exp \left(-\frac{\|x - z\|^2}{\sigma^2}\right)F(dz) \ge \frac{K_2}{\sigma^d}F(U_{J(x)}) \ge K_3\frac{\tilde\epsilon_n^{4db_1}}{\sigma^d}
\]
for some constants $K_1, K_2, K_3$, where $J(x)$ denotes the index $j\in\{1, \ldots, K\}$ for which $x \in U_j$. The penultimate inequality follows because $U_{J(x)}$ with diameter no larger than $\sigma$ must be a subset of a ball of radius $\sigma$ around $x$. Also, for any $x \in\mathbb{R}^d$ with $\|x\| > a_\sigma$,
\[
\frac{p_{F,\Sigma}(x)}{f_0(x)} \ge \frac{K_1}{\sigma^d} \int_{\|z\| \le a_\sigma} \exp \left(-\frac{\|x - z\|^2}{\sigma^2}\right)F(dz) \ge \frac{K_4}{\sigma^d} \exp(-4\|x\|^2 / \sigma^2)
\]
for some constant $K_4$ because $\|x - z\|^2 \le 2 \|x\|^2 + 2\|z\|^2 \le 4 \|x\|^2$ and $F(\{x\in\mathbb{R}^d : \|x\| \le a_\sigma\}) \ge 1 - 2\tilde\epsilon_n^{2db_1}$. Set $\lambda = K_3\tilde\epsilon_n^{4db_1} / \sigma^d$ and notice $\log(1/\lambda) \le K_5 \log(1/\tilde\epsilon_n)$ for some constant $K_5$. For any $F \in \mathcal{P}_\sigma$, $\Sigma \in \mathcal{S}_\sigma$,
\begin{align*}
P_0\left\{\left(\log\frac{f_0}{p_{F,\Sigma}}\right)^2\Ind\left(\frac{p_{F,\Sigma}}{f_0} < \lambda\right)\right\} & \le \frac{K_6}{\sigma^4}\int_{\|x\| > a_\sigma} \|x\|^4f_0(x)dx\\
& \le \frac{K_6}{\sigma^4} (P_0\|X\|^8)^{1/2}P_0(E_\sigma^c)^{1/2} \le K_7\sigma^{2\beta + \epsilon}
\end{align*}
for some constant $K_7$ since $P_0\|X\|^m < \infty$ for all $m > 0$ because of the tail condition \eqref{eq:tail}.
Given $n$ sufficiently large, we have $\lambda < e^{-1}$ and hence $\log({f_0}/{p_{F,\Sigma}})\Ind({p_{F,\Sigma}}/{f_0}< \lambda) \leq \left\{\log({f_0}/{p_{F,\Sigma}})\right\}^2\Ind({p_{F,\Sigma}}/{f_0}< \lambda)$.
Therefore $P_0\{\log({f_0}/{p_{F,\Sigma}})\Ind({p_{F,\Sigma}}/{f_0}< \lambda)\} \le K_7\sigma^{2\beta + \epsilon}$. Now apply Lemma \ref{lem:kl by hell} to conclude that both $P_0\{\log (f_0 / p_{F,\Sigma})\}$ and $P_0\{\log (f_0 / p_{F,\Sigma})\}^2$ are bounded by $K_8 \log(1/\lambda)^2 \sigma^{2\beta} \le K_9\sigma^{2\beta}\{\log(1/\tilde\epsilon_n)\}^2 \le A\tilde\epsilon_n^2$ for some positive constant $A$. Therefore
\begin{align*}
(\mathcal{D}_\alpha \times G)& \left[P_0\log\frac{f_0 }{p_{F,\Sigma}} \le A\tilde \epsilon_n^2, P_0\left(\log\frac{f_0 }{ p_{F,\Sigma}}\right)^2 \le A\tilde \epsilon_n^2\right] \\
& \ge \mathcal{D}_\alpha (\mathcal{P}_\sigma)G(\mathcal{S}_\sigma) \nn \\
& \ge C_4\exp\left[-c_4\tilde\epsilon_n^{-\dstar /\beta}\{\log(1/\tilde\epsilon_n)\}^{s\dstar + 1}\right].
\end{align*}
This gives \eqref{eq:inthick} provided $\tilde\epsilon_n^{-\dstar/\beta}\{\log(1/\tilde\epsilon_n)\}^{s\dstar + 1} \le n\tilde\epsilon_n^2.$
With $\tilde \epsilon_n = n^{-\beta / (2\beta + \dstar)}(\log n)^t$, the condition is satisfied if $t \ge (s\dstar + 1) / (2 + \dstar / \beta)$.
%we get $n\tilde\epsilon_n^2 = n^{\dstar/(2\beta + \dstar)}(\log n)^{2t}$ whereas $\tilde\epsilon_n^{-\dstar/\beta}\{\log(1/\tilde\epsilon_n)\}^{s\dstar + 1} \le n^{\dstar/(2\beta + \dstar)}(\log n)^{s\dstar + 1 - \dstar t / \beta}$ and hence
\end{proof}

\section{Sieve construction}
%The chief novelty of the sieve proposed in Pati et al.'s unpublished work lies in exploiting the stick-breaking representation of a Dirichlet process distribution.
In the following proposition, based on the stick-breaking representation of a Dirichlet process, we give an explicit definition of the sieve and derive upper bounds for its entropy and the prior probability of its complement. This result serves as the main tool in obtaining adaptive posterior convergence rates; a proof is given in Appendix A.

\begin{proposition}
\label{basic sieve}
Fix $\epsilon, a, \sigma_0 > 0$ and integers $M, H \ge d$. Define
\begin{equation}
\label{eff}
\scQ = \left\{p_{F, \Sigma} ~\mbox{with}~ F = \sum_{h = 1}^\infty \pi_h \delta_{z_h}: \begin{array}{l}  z_h \in [-a,a]^d, h \le H;~~\sum_{h > H} \pi_h < \epsilon;\\[5pt]  \sigma_0^2 \le \eig_j(\Sigma) < \sigma_0^2\left(1 + {\epsilon^2}/d\right)^M,~~j = 1, \ldots, d\end{array}\right\}.
\end{equation}
Then:
\begin{enumerate}
\item $\log N(\epsilon, \scQ, \rho) \le K \left\{ dH \log\left\{a/(\sigma_0\epsilon)\right\} - H \log \epsilon + \log M + M\epsilon^2 \right\}$, for some constant $K$, where $\rho$ is either the Hellinger or the $L_1$-metric;
\item $(\mathcal{D}_\alpha\times G)(\scQ^c) \leq b_1 H \exp\{-C_1 a^{a_1}\} + \{(e |\alpha|/H)\log(1/\epsilon)\}^H  + b_2\exp \{-C_2\sigma_0^{-2a_2}\} + b_3\sigma_0^{-2a_3}(1 + \epsilon^2/d)^{-2Ma_3}$, with the constants as defined in \eqref{eq:dp1}--\eqref{eq:iw3}.
\end{enumerate}
\end{proposition}

The sieve defined here can easily adapt to different rates of convergence of the form $\epsilon_n = n^{-\gamma}(\log n)^{(d + 1 + s)/2}$ for $0 < \gamma \leq 1/2$ and $s > 0$.
The extreme case $\gamma = 1/2$ corresponds to the class of Gaussian mixtures \citep{Ghosal20011}. %This rate could be obtained by constructing a $\mathcal{P}_n$ as $Q$ in \eqref{eff} with $\epsilon = n^{-1/2}(\log n)^{(d + 1 + s)/2}$, $H = \lfloor (\log n)^{d + s}\rfloor$, $M = a^2 = \underline{\sigma}^{-2} = n$ which gives $\Pi (\mathcal{P}_n^c) \leq \exp\{-c n \epsilon_n^2\}$ and $\log N(\epsilon,\mathcal{P}_n, \rho) \le n\epsilon_n^2$ for every $c>0$.
For a $\beta$-H\"older class convergence rate we need to work with $\gamma = \beta / (2\beta + \dstar)$. The following theorem makes this precise.

\begin{theorem}
Fix a $\gamma \in (0, 1/2)$ and a pair of numbers $t > t_0 \ge (d + 1)/2$. For $n \ge 1$, take $\epsilon_n = n^{-\gamma}(\log n)^{t}$, $\tilde \epsilon_n = n^{-\gamma}(\log n)^{t_0}$ and define $\mathcal{F}_n$ as $\mathcal{Q}$ in \eqref{eff} with $\epsilon = \epsilon_n$, $H = \lfloor n\epsilon_n^2 / (\log n)\rfloor$, $M = a^{a_1} = \sigma_0^{-2a_2} = n$. Then $\mathcal{F}_n$ satisfies \eqref{eq:e28} and \eqref{eq:e30} for all large $n$ for some $c_1, c_3 > 0$ and every $c_2 > 0$.
\label{thm:sieve}
\end{theorem}

\begin{proof}
By Proposition \ref{basic sieve},
\begin{align*}
\log N(\bar \epsilon_n, \mathcal{F}_n, \rho) & \leq K\{ d n^{1 - 2\gamma}(\log n)^{2t} + n^{1 - 2\gamma}(\log n)^{2t} + \log n + n^{1 - 2\gamma}(\log n)^{2t}\}\\
& \leq c_1 n^{1 - 2\gamma}(\log n)^{2t} = c_1 n \epsilon_n^2
\end{align*}
for some $c_1 > 0$ and hence \eqref{eq:e28} holds. By the second assertion of the same proposition,
\begin{align*}
\label{step1b}
(\mathcal{D}_\alpha \times G)(\mathcal{F}_n^c) & \leq b_1n^{1 - 2\gamma}(\log n)^{2t - 1}e^{-b_1n} + n^{-(1 - 2\gamma)n^{1 - 2\gamma}(\log n)^{2t - 1}} \\
& ~~~~~+~~ b_2 e^{-C_2 n} + b_3 n^{a_3/a_2}e^{-2a_3 n\log(1 + \bar\epsilon_n^2/d)}\\
& \leq c_3 e^{-(1 - 2\gamma)n^{1 - 2\gamma}(\log n)^{2t}} \leq c_3 e^{-(c_2 + 4)n^{1 - 2\gamma}(\log n)^{2t_0}}
\end{align*}
for all large $n$, some $c_3 > 0$ and every $c_2 > 0$.
\end{proof}

\section{Anisotropic H\"older functions}

Anisotropic functions are those that have different orders of smoothness along different axes. The isotropic result presented before gives adaptive rates corresponding to the least smooth direction. Sharper results can be obtained by explicitly factoring in the anisotropy. For any $a = (a_1, \ldots, a_d)$ and $b = (b_1, \ldots, b_d)$, let $\langle a, b \rangle$ denote $a_1b_1 + \cdots + a_db_d$ and for $y = (y_1, \ldots, y_d)$, let $\|y\|_1$ denote the $L_1$-norm $|y_1| + \cdots + |y_d|$. For a $\beta > 0$, an $\alpha = (\alpha_1, \ldots, \alpha_d) \in (0, \infty)^d$ with $\alpha. = d$ and an $L:\mathbb{R}^d \to (0, \infty)$ satisfying $L(x + y) \le L(x) \exp(\tau_0 \|y\|_1^2)$ for all $x, y \in \mathbb{R}^d$ and some $\tau_0 > 0$, the $\alpha$-anisotropic $\beta$-H\"older class with envelope $L$ is defined as the set of all functions $f:\mathbb{R}^d \to \mathbb{R}$ that have continuous mixed partial derivatives $D^k f$ of all orders $k \in \natz^d$, $\beta - \alphamax \le \langle k, \alpha \rangle < \beta$, with
\[
|D^kf(x + y) - D^k f(x)| \le L(x)e^{\tau_0 \|y\|_1^2}\sum_{j = 1}^d |y_j|^{\min(\beta / \alpha_j - k_j, 1)},~~x,y \in \mathbb{R}^d,
\]
where $\alphamax = \max(\alpha_1, \ldots, \alpha_d)$. We denote this set of functions by $\mathcal{C}^{\alpha,\beta,L,\tau_0}(\mathbb{R}^d)$.
Here $\beta$ refers to the mean smoothness and $\alpha$ refers to the anisotropy index.
An $f \in \mathcal{C}^{\alpha, \beta, L, \tau_0}$ has partial derivatives of all orders up to $\lfloor\beta_j\rfloor$ along axis $j$
where $\beta_j = \beta / \alpha_j$, and $\beta$ is the harmonic mean $d / (\beta_1^{-1}  + \cdots + \beta_d^{-1})$ of these axial smoothness coefficients. In the special case of $\alpha = (1, \ldots, 1)$, the anisotropic set $\mathcal{C}^{\alpha, \beta, L, \tau_0}(\mathbb{R}^d)$ equals the isotropic set $\mathcal{C}^{\beta, L, \tau_0}(\mathbb{R}^d)$.

\begin{theorem}\label{thm:x3}
Suppose that $f_0 \in \mathcal{C}^{\alpha, \beta, L, \tau_0}(\mathbb{R}^d)$ is a probability density function satisfying
\begin{equation}
P_0\left(|D^kf_0|/f_0\right)^{{(2\beta + \epsilon)}/{\langle k, \alpha \rangle}} < \infty,~~k\in \natz^d, \langle k,\alpha\rangle < \beta,~~~~P_0\left(L/f_0\right)^{(2\beta + \epsilon)/\beta}  < \infty
\label{eq:int3}
\end{equation}
for some $\epsilon > 0$ and that \eqref{eq:tail} holds for some constants $a, b, c, \tau>0$.
If $\Pi$ is as in Section \ref{sec: prior}, then the posterior convergence rate at $f_0$ in the Hellinger or the $L_1$-metric is $\epsilon_n = n^{-\beta / (2\beta + \dstar)} (\log n)^t$, where $t \geq \{ \dstar(1 + \tau^{-1} + \beta^{-1}) + 1\} /(2 + \dstar/\beta)$, and $\dstar = \max(d, \kappa \alphamax)$.
\end{theorem}

A proof, given in Appendix A, is similar to those of the results presented in Section \ref{sec3}, except that to obtain an approximation to $f_0$, we replace the single bandwidth $\sigma$ with bandwith $\sigma^{\alpha_j}$ along the $j$-th axis. An $f_0$ satisfying the conditions of the above theorem also satisfies the conditions of Theorem \ref{thm:x2} with smoothness index $\beta /  \alphamax$, which is strictly smaller than $\beta$ as long as not all $\alpha_j$ are equal to 1. Therefore when the true density is anisotropic, Theorem \ref{thm:x3} indeed leads to a sharper convergence rate result.

 With the standard inverse Wishart prior $G$, we have $\kappa = 2$ and consequently the optimal rate $n^{-\beta / (2\beta + d)}$ is recovered up to a $\log n$ factor only when $\alphamax \le d / 2$. Therefore in a two dimensional case only the isotropic case is addressed and for higher dimensions we get optimal results for a limited amount of anisotropy. But, when $\kappa \le 1$, as in the case of a diagonal $\Sigma$ with squared inverse gamma diagonal components, Theorem \ref{thm:x3} provides optimal rates for any dimension and any degree of anisotropy because $\alphamax$ can never exceed $d$.

\section*{Acknowledgements}
The authors thank the editor, the associate editor and the reviewers for comments that substantially improved
the paper. Research of the first and the third author is partially supported by NSF grant number DMS-1106570.

\appendix

\section*{Appendix A. Proofs}

\begin{proof}[of Lemma \ref{lem:inw}]
Let $\Sigma \sim \textit{IW}(\nu, \Psi)$ and suppose $\Psi = I$. It is well known that $\mathrm{tr}(\Sigma^{-1}) \sim \chi^2_{\nu d}$, the chi-square distribution with $\nu d$ degrees of freedom. The cumulative distribution function $F(x; k)$ of $\chi^2_k$ satisfies $1 - F(zk; k) \le \{z\exp(1 - z)\}^{k/2}$ for all $z > 1$. Therefore for all $x > \nu d$,
$$\pr\left\{\eig_d(\Sigma^{-1}) > x\right\} \le \pr\left\{\mathrm{tr}(\Sigma^{-1}) > x\right\}\le \left(\frac{x}{\nu d}\right)^{\nu d / 2}\exp\{(\nu d - x)/2\} \leq b_2 \exp(-C_2x)$$
for some constants $b_2, C_2$. Furthermore, the joint probability density of $\eig_1(\Sigma^{-1}), \ldots,  \eig_d(\Sigma^{-1})$ is
\begin{equation}
\label{eq:jpdf}
f(x_1, \ldots, x_d) = c_{d,\nu}~e^{- \sum_j x_j / 2} \prod_{j = 1}^d x_j^{(\nu + 1 - d)/2}\prod_{j < k}(x_k - x_j)
\end{equation}
over the set $\{(x_1, \ldots, x_d) \in (0, \infty)^d: x_1 \le \cdots \le x_d\}$, for a known constant $c_{d,\nu}$. Since $\prod_{j < k} (x_k - x_j) \le \prod_{j < k} x_k = \prod_{k = 2}^d x_k^{k-1}$, the probability density of $\eig_1(\Sigma^{-1})$ satisfies
$$f (x_1) \le c_{d,\nu}x_1^{(\nu  + 1 - d)/2} e^{-x_1/2} \prod_{k = 2}^d \left\{ \int_0^\infty  x_k^{(\nu + 1 - d)/2 + k - 1} e^{-x_k/2}dx_k\right\} = \tilde c_{d,\nu}x_1^{(\nu + 1 - d)/2} e^{-x_1 / 2}$$
for all $x_1 > 0$ and some positive constant $\tilde c_{d,\nu}$. Therefore for any $x > 0$
\[
\pr\left\{\eig_1(\Sigma^{-1})  < x \right\} \le \tilde c_{d,r} \int_0^{x} x_1^{(\nu + 1 - d)/2}dx_1 \leq b_3 x^{a_3}
\]
for some positive constants $a_3, b_3$.

Next, notice that the set on the left hand side of \eqref{eq:iw3} contains all $\Sigma$
which have $\eig_j(\Sigma^{-1}) \in I_j = \left(s_j \{1 + (j - 1/2) t / d\}, s_j(1 + jt / d)\right)$  $(j = 1, \ldots, d)$
and for any positive integers $k > j$, $x_j \in I_j$ and $x_k \in I_k$ implies that $x_k - x_j > s_k \{1 + (k - 1/2) t / d\} - s_j(1 + jt / d) \geq s_1t / (2d)$. Therefore
\begin{align*}
\pr \{s_j & < \eig_j(\Sigma^{-1})  < s_j(1 + t), ~~j = 1, \ldots, d\}  \\
 &\ge \int_{I_d}\cdots \int_{I_1}   c_{d,\nu}~\exp\left(- \sum_j x_j / 2\right) \prod_{j = 1}^d x_j^{(\nu + 1 - d)/2}\prod_{j < k}(x_k - x_j)dx_1\cdots d x_d\\
 &\ge  c_{d,\nu} \exp\left(-d s_d\right)s_1^{d(\nu + 1 - d)/2} \{t/(2d)\}^{d(d - 1)/2}\int_{I_d}\cdots \int_{I_1}   dx_1\cdots d x_d\\
 &=  c_{d,\nu} \exp\left(-d s_d\right)s_1^{d(\nu + 1 - d)/2} \{t/(2d)\}^{d(d - 1)/2} \{s_1t/(2d)\}^d,
\end{align*}
which gives \eqref{eq:iw3} for some positive constants $a_4,a_5,b_4, C_3$.

If $\Psi \neq I$, applying the above results for $\Psi^{-1} \Sigma \sim \text{IW}(\nu,I)$, the conclusion holds for a different set of constants.
\end{proof}

\begin{proof}[of Lemma \ref{lem:basic}]
From multivariate Taylor expansion of any $f \in \mathcal{C}^{\beta, L, \tau_0}(\mathbb{R}^d)$
\[
f(x - y) - f(x) = \sum_{1 \le k_\cdot \le \lfloor \beta \rfloor}\frac{(-y)^{k_\cdot}}{k!}(D^kf)(x) + R(x, y)
\]
with the residual satisfying $|R(x, y)| \le K_1L(x)\exp(\tau_0 \|y\|^2)\|y\|^\beta$ for every $x, y \in \mathbb{R}^d$ and for a universal constant $K_1$. Therefore for any $\sigma \in (0, 1/(2\tau_0)^{1/2})$,
\begin{align}
\{K_\sigma( & T_{\beta,\sigma}f)  - f\}(x)\nn\\
& = \int \phi_{\sigma^2 I}(y)\{f(x - y) - f(x)\}dy - \sum_{2\le k_\cdot \le \lfloor\beta\rfloor}d_k \sigma^{k_\cdot}\{K_\sigma (D^kf)\}(x)\nn\\
& = \int \phi_{\sigma^2 I}(y)R(x, y)dy + \sum_{2\le k_\cdot \le \lfloor\beta\rfloor} \sigma^{k_\cdot}\left[\frac{(-1)^{k_\cdot} m_k}{k!} (D^kf)(x) - d_k\{K_\sigma (D^k f)\}(x)\right] \label{eq:decomp}.
\end{align}
The first term of \eqref{eq:decomp} is bounded by $K_2L(x)\sigma^\beta$ for some universal constant $K_2$. If $\beta \le 2$ then the second term of \eqref{eq:decomp} does not exist and we get a proof with $M_\beta = K_2$. For $\beta > 2$ we use induction on $\lfloor \beta \rfloor$.

From \eqref{eq:cd} we can rewrite the second term of \eqref{eq:decomp} as
\begin{equation}
\sum_{2\le k_\cdot \le\lfloor\beta\rfloor}\left[\frac{(-1)^{k_\cdot}m_k \sigma^{k_\cdot}}{k!}\{D^k f - K_\sigma (D^kf)\}(x) - c_k \sigma^{k_\cdot}\{K_\sigma (D^k f)\}(x)\right]
\label{eq:decomp2}.
\end{equation}
For each $1 \le k_\cdot \le \lfloor\beta\rfloor$, the induction hypothesis implies that $D^k f \in \mathcal{C}^{\beta - k_\cdot, L, \tau_0}(\mathbb{R}^d)$ and
\begin{equation}
D^kf - K_\sigma (D^k f) = \{D^kf - K_\sigma T_{\beta - k_\cdot, \sigma} (D^k f)\} + K_\sigma \{T_{\beta - k_\cdot, \sigma} (D^k f) - D^k f\}
\label{eq:decom3}
\end{equation}
with
$
| \{D^kf - K_\sigma T_{\beta - k_\cdot, \sigma} (D^k f)\} (x)| \le M_{\beta - k_\cdot}L(x)\sigma^{\beta - k_\cdot},~~\mbox{for all}~~x \in \mathbb{R}^d.
$
This establishes the claim with $M_\beta = K_2 + \sum_{2\le k_\cdot \le \lfloor\beta\rfloor} (m_k / k!) M_{\beta - k_\cdot}$ because
\begin{align*}
\sum_{2 \le k_\cdot \le \lfloor\beta\rfloor}&\left[\frac{(-1)^{k_\cdot}m_k\sigma^{k_\cdot}}{k!}\{T_{\beta - k_\cdot, \sigma}(D^kf) - D^kf\} - c_k\sigma^{k_\cdot} D^kf\right] \\
& = \sum_{2 \le k_\cdot \le \lfloor\beta\rfloor}\left\{\frac{(-1)^{k_\cdot}m_k\sigma^{k_\cdot}}{k!} \sum_{1 \le j_\cdot \le \lfloor\beta\rfloor - k_{\cdot}} d_j\sigma^{j_\cdot}D^{k + j}f- c_k\sigma^{k_\cdot} D^kf\right\}\\
& = \sum_{3 \le n_\cdot \le \lfloor\beta\rfloor} \left\{\sum_{\substack{n\,=\,l + k\\l_\cdot \ge 1, k_\cdot \ge 2}} \frac{(-1)^{k_\cdot}}{k!} m_kd_l - c_n\right\}\sigma^{n_\cdot} D^n f  = 0
\end{align*}
identically, by definitions of $c_n$ and $d_n$.
\end{proof}

\begin{proof}[of Theorem \ref{thm:hell approx}]
Fix $s_0 \in (0, 1/(2\tau_0)^{1/2})$ such that $\sum_{1\le k_\cdot \le \lfloor\beta\rfloor} |d_k||\log \sigma|^{-k_\cdot/2} < 1/2$ and $\sigma^\epsilon|\log \sigma|^{(2\beta + \epsilon)/2} < 1$ for all $0 < \sigma < s_0$. For any $\sigma \in (0, s_0)$ define
\[
A_\sigma = \left\{x : \frac{|D^k f_0(x)|}{f_0(x)} \le \sigma^{-k_\cdot}|\log \sigma|^{-k_\cdot/2},k_\cdot \le \lfloor\beta\rfloor,~~\frac{L(x)}{f_0(x)} \le \sigma^{-\beta}|\log \sigma|^{-\beta / 2}\right\}
\]
and notice that by Markov's inequality
\begin{align*}
P_0(A_\sigma^c) & \le \sum_{k_\cdot \le \lfloor\beta\rfloor}P_0\left\{\frac{|D^k f_0(X)|}{f_0(X)} > \sigma^{-k_\cdot}|\log \sigma|^{-k_\cdot/2}\right\} + P_0\left\{\frac{L(X)}{f_0(X)} > \sigma^{-\beta}|\log \sigma|^{-\beta/2}\right\}\\
& =  \sum_{k_\cdot \le \lfloor\beta\rfloor}P_0\left\{\left(|D^k f_0|/f_0\right)^{{(2\beta + \epsilon)}/{k_\cdot}} > \sigma^{-(2\beta + \epsilon)}|\log \sigma|^{-{(2\beta + \epsilon)}/2}\right\}\\
& ~~~~~~~~~~~~+P_0\left\{\left(L/f_0\right)^{{(2\beta + \epsilon)}/{\beta}} > \sigma^{-(2\beta + \epsilon)}|\log \sigma|^{-{(2\beta + \epsilon)}/2}\right\}\\
& \le \sigma^{2\beta + \epsilon} |\log\sigma|^{{(2\beta + \epsilon)}/2} \left\{\sum_{k_\cdot \le \lfloor\beta\rfloor}P_0\left(|D^kf_0|/f_0\right)^{{(2\beta + \epsilon)}/{k_\cdot}} +  P_0\left(L/f_0\right)^{{(2\beta + \epsilon)}/{\beta}} \right\},
\end{align*}
which is bounded by $K_1\sigma^{2\beta}$ for some constant $K_1$. Also, for any $x \in A_\sigma$,
\[
|(f_{\sigma} - f_0)(x)| \le \sum_{1\le k_\cdot \le \lfloor\beta\rfloor} |d_k|\sigma^{k_\cdot}|D^kf_0(x)| \le f_0(x)\sum_{1\le k_\cdot \le \lfloor\beta\rfloor} |d_k||\log \sigma|^{-k_\cdot/2} \le \frac{1}{2}f_0(x).
\]
Consequently, $f_{\sigma} \ge f_0 / 2$ on $A_\sigma$. Because of integrability conditions on $D^kf_0 / f_0$, it turns out that in calculating $\int D^k f_0(x) dx$ for any $1 \le k_\cdot \le \lfloor\beta\rfloor$, one can integrate under the derivative and conclude that $\int D^k f_0(x) dx = 0$ as $f_0$ is a density. So $\int f_{\sigma}(x) dx = 1$ and for some constant $K_2$ and all $\sigma < s_0$,
\[
1 \le \int g_{\sigma}(x)dx \le 1 + \frac12 \int f_0(x)\Ind\{f_{\sigma}(x) < f_0(x) / 2\}dx \le 1+ \frac12P_0({A_\sigma^c}) \le 1 + K_2\sigma^{2\beta}.
\]
So $\int g_{\sigma}(x)dx < \infty$ and $h_{\sigma}$ is a well defined probability density function on $\mathbb{R}^d$.

To prove the final result of Theorem \ref{thm:hell approx}, denote $r_{\sigma} =\frac12f_0 \,\Ind\{f_{\sigma} < (1/2)f_0\}$
and $c_{\sigma} = \int g_{\sigma}(x)dx$ and note that for $a > 0, b > 0$, we have
$
(a^{1/2} -  b^{1/2})^2 = (a - b)^2/(a^{1/2} +  b^{1/2})^2 \le (a - b)^2/(a + b)
$
and hence
\begin{align}
%d^2_H&(f_0, K_\sigma h_{\sigma}) \nn \\
d^2_H(f_0, K_\sigma h_{\sigma})  & \le \int \frac{(f_0 - K_\sigma h_{\sigma})^2(x)}{f_0(x) + (K_\sigma h_{\sigma})(x)}dx \nn\\
& = \frac1{c_\sigma} \int \frac{(c_{\sigma}f_0 - K_\sigma g_{\sigma})^2(x)}{c_{\sigma} f_0(x) + (K_\sigma g_{\sigma})(x)}dx\nn\\
& \le3 \int \frac{(c_{\sigma} - 1)^2f_0^2(x) + (f_0 - K_\sigma f_{\sigma})^2(x) + (K_\sigma r_{\sigma})^2(x)}{c_{\sigma}f_0(x) + (K_\sigma g_{\sigma})(x)}dx\nn\\
& \le 3\left\{\int (c_{\sigma} - 1)^2 f_0(x)dx + \int \frac{(f_0 -K_\sigma f_{\sigma})^2(x)}{f_0(x)}dx + \int \frac{(K_\sigma r_{\sigma})^2(x)}{(K_\sigma g_{\sigma})(x)}dx\right\}\nn\\
& \le 3\left\{K_2^2 \sigma^{4\beta} + M_\beta^2 \sigma^{2\beta} P_0\left({L}/{f_0}\right)^2 + \int (K_\sigma r_{\sigma})(x) dx\right\}, \label{eq:3term}
\end{align}
because $1 \le c_{\sigma} \le 1 + K_2\sigma^{2\beta}$, $|(f_0 - K_\sigma f_{\sigma})(x)| < M_\beta L(x)\sigma^\beta$ and $K_\sigma r_{\sigma} \le K_\sigma g_{\sigma}$ since $r_{\sigma} \le g_{\sigma}$. By Jensen's inequality $P_0(L / f_0)^2 \le \{P_0(L / f_0)^{(2\beta + \epsilon) / \beta}\}^{\beta / (\beta + \epsilon/2)} < \infty$. Also, $\int(K_\sigma r_{\sigma})(x)dx$ is
\begin{align*}
\frac12\int\int \phi_{\sigma^2 I}(x - y) f_0(y)\Ind\{f_{\sigma}(y) < f_0(y)/2\}dxdy
= \frac12 \int f_0(y)\Ind\{f_{\sigma}(y) < f_0(y)/2\}dy,
\end{align*}
which is bounded by  $P_0(A_\sigma^c) \le K_1\sigma^{2\beta}$.
\end{proof}

\begin{proof}[of Proposition \ref{prop:compact}]
Define $g_\sigma$ and $h_\sigma$ as in the statement of Theorem \ref{thm:hell approx}.
This theorem implies that there are $s_1, K > 0$ such that
$d^2_H(f_0, K_\sigma h_\sigma) \le K \sigma^{2\beta}$ for all $0 < \sigma < s_1$.
The tail condition on $f_0$ implies existence of a small $\delta > 0$ such that $B_0$, which is defined as $P_0(f_0^{-\delta})$,
satisfies $B_0 < \infty$. Let $s_2 \in (0, 1/(2\tau_0)^{1/2})$ be such that $\{(4\beta + 2\epsilon + 8)/(b\delta)\} \log (1/s_2) > \max\{(1/b)\log c, a^\tau/2\}$. Set $s_0 = \min(s_1, s_2)$ and pick any $\sigma \in (0, s_0)$. Define $E_\sigma = \{x \in \mathbb{R}^d: f_0(x) \ge \sigma^{(4\beta + 2\epsilon + 8)/\delta}\}$ and $a_\sigma = a_0 \log (1/\sigma)^{1/\tau}$ with $a_0 = \{(8\beta + 4\epsilon + 16) / (b\delta)\}^{1/\tau}$. Then $a_\sigma > a$ and $E_\sigma \subset \{x\in \mathbb{R}^d: \|x\| \le a_\sigma\}$.

By Markov's inequality,
$
P_0(E_\sigma^c) = P_0\{f_0(X)^{-\delta} > \sigma^{-(4\beta + 2\epsilon + 8)}\} \le B_0\sigma^{4\beta + 2\epsilon + 8}  \le B_0 \sigma^{2\beta + \epsilon}
$
and consequently by \eqref{eq:int} and applications of H\"older's inequality
\begin{align*}
\int_{E_\sigma^c} g_{\sigma}(x)dx & \le \frac32\int_{E_\sigma^c} f_0(x)dx + \sum_{k_\cdot = 1}^{\lfloor\beta\rfloor} \sigma^{k_\cdot}|d_k| \int_{E_\sigma^c} |D^k f_0(x)|dx \\
& \le \frac32 P_0(E_\sigma^c) + \sum_{k_\cdot = 1}^{\lfloor\beta\rfloor} \sigma^{k_\cdot}|d_k| \left\{P_0\left(|D^kf_0|/f_0\right)^{{(2\beta + \epsilon)}/{k_\cdot}}\right\}^{{k_\cdot}/{2\beta + \epsilon}}P_0(E_\sigma^c)^{{(2\beta + \epsilon - k_\cdot)}/{(2\beta + \epsilon)}},
\end{align*}
which is bounded by $B_1 \sigma^{2\beta + \epsilon}$
for some constant $B_1$ that does not depend on $\sigma$. Hence $\int_{E_\sigma^c} h_\sigma(x) dx \le \int_{E_\sigma^c} g_\sigma(x)dx \le B_1\sigma^{2\beta + \epsilon}$.

Define $\tilde h_\sigma$ to be the restriction of $h_\sigma$ to $E_\sigma$, i.e., $\tilde h_\sigma(x) = h_\sigma(x)\Ind(x \in E_\sigma)/ \int_{E_\sigma} h(x)dx$. Then $d_H(K_\sigma h_\sigma, K_\sigma \tilde h_\sigma) \le d_H(h_\sigma, \tilde h_\sigma) = [2 - 2\{\int_{E_\sigma} h_\sigma(x)dx\}^{1/2}]^{1/2} = O(\sigma^{\beta + \epsilon / 2})$. This completes the proof because $d_H(f_0, K_\sigma\tilde h_\sigma) \le d_H(f_0, K_\sigma h_\sigma) + d_H(K_\sigma h_\sigma, K_\sigma \tilde h_\sigma).$
\end{proof}

\begin{proof}[of Proposition \ref{basic sieve}]
Let $\hat R$ be a $(\sigma_0\epsilon)$-net of $[-a,a]^d$, $\hat S$ be an $\epsilon$-net of the $H$-simplex $\S_H = \{p = (p_1,  \ldots, p_H): p_h \ge 0, \sum_{h=1}^H p_h = 1\}$ and $\hat \O$ be an $\delta$-net of $\O_d$, the group of $d\times d$ orthogonal matrices equipped with the spectral norm $\|\cdot\|_2$, where $\delta = \epsilon^2 / \{3d(1 + \epsilon^2 / d)^M\}$. It is well known that the cardinality of $\hat R \lesssim \{a/(\sigma_0\epsilon)\}^d$, that of $\hat S$ is $\lesssim \epsilon^{-H}$ and that of $\hat \O \lesssim \delta^{-d(d - 1)/2}$.

Pick any $p_{F, \Sigma} \in \scQ$, with $F = \sum_{h = 1}^\infty z_h \delta_{z_h}$ and let the spectral decomposition of $\Sigma^{-1}$ be $P\Lambda P^T$
where $\Lambda = \diag(\lambda_1, \ldots , \lambda_d)$ and $P$ is an orthogonal matrix. Find $\hat z_1, \ldots, \hat z_H \in \hat R$,
$\hat\pi = (\hat\pi_1, \ldots, \hat\pi_H) \in \hat S$, $\hat P \in \hat \O$ and $\hat m_1, \ldots, \hat m_d \in \{1, \ldots, M\}$ such that
\begin{align}
& \max_{1 \le h \le H}\|z_h - \hat z_h\| < \sigma_0 \epsilon,\\
%\item
& \sum_{h = 1}^H |\tilde \pi_h - \hat\pi_h| < \epsilon,\mbox{ where }\tilde \pi_h = \frac{\pi_h }{ 1-\sum_{l > H} \pi_l}, 1 \le h \le H,\\
%\item
&  \|P - \hat P\|_2 \le \epsilon^2, \\
& \hat\lambda_j = \{\sigma_0^2(1 + \epsilon^2/d)^{\hat m_j - 1}\}^{-1}\mbox{ satisfies }1 \le \hat\lambda_j/\lambda_j < 1 + \epsilon^2/d,~~ j = 1, \ldots, d.
\end{align}
%\end{enumerate}
Take $\hat F = \sum_{h = 1}^H \hat\pi_h \delta_{\hat z_h}$ and $\hat \Sigma = (\hat P\hat\Lambda \hat P^T)^{-1}$ where $\hat \Lambda = \mathrm{diag}(\hat\lambda_1, \ldots, \hat\lambda_d)$. Also define $\tilde \Sigma = (\hat P \Lambda \hat P^T)^{-1}$ and $Q = \hat P^TP$. By the triangle inequality
\begin{equation}
\label{tri}
\|p_{F, \Sigma} - p_{\hat F, \hat \Sigma}\|_1 \le \|p_{F, \Sigma} - p_{F, \hat \Sigma}\|_1 + \|p_{F, \hat\Sigma} - p_{\hat F, \hat \Sigma}\|_1.
\end{equation}
The first term on the right hand side can be bounded by
\begin{align*}
 \int \|\phi_\Sigma(\cdot - z) - \phi_{\hat\Sigma}(\cdot - z)\|_1dF(z) =  \|\phi_\Sigma - \phi_{\hat\Sigma}\|_1 \le \|\phi_\Sigma - \phi_{\tilde \Sigma}\|_1 + \|\phi_{\tilde\Sigma} - \phi_{\hat\Sigma}\|_1.
\end{align*}
Since the total variation distance is bounded by $2^{1/2}$ times the square root of the Kullback--Leibler divergence, we have $\|\phi_{\tilde \Sigma} - \phi_{\hat \Sigma}\|_1  \le \{\mathrm{tr}(\hat\Sigma^{-1}\tilde\Sigma)  - \log \det (\hat\Sigma^{-1}\tilde \Sigma) - d\}^{1/2}$. But $\Tr (\hat\Sigma^{-1}\tilde\Sigma) = \Tr(\hat \Lambda\Lambda^{-1}) = \sum_{j = 1}^d \hat \lambda_j / \lambda_j < d + \epsilon^2$ and $\det (\hat\Sigma^{-1}\tilde\Sigma) = \prod_{j = 1}^d (\hat\lambda_j / \lambda_j) > 1$. Thus $\|\phi_{\tilde\Sigma} - \phi_{\hat\Sigma}\|_1 \le \epsilon$. For the other term, we have
$\|\phi_\Sigma - \phi_{\tilde\Sigma}\|_1 \le \{\mathrm{tr}(\Sigma^{-1}\hat\Sigma) - \log\det(\Sigma^{-1}\hat\Sigma) - d\}^{1/2} = \{\mathrm{tr}(Q\Lambda Q^T\Lambda^{-1} - I)\}^{1/2}$
because $\Sigma^{-1}\hat\Sigma = P\Lambda P^T\hat P \Lambda^{-1} \hat P^T$ has determinant one and trace equal to that of $Q\Lambda Q^T\Lambda^{-1}$. Write $Q = I + B$. Then $\|B\|_{\max} \le \|B\|_2 = \|\hat P^T P - I\|_2 = \|P - \hat P\|_2 \leq \delta$ and hence
$$\mathrm{tr}(Q\Lambda Q^T\Lambda^{-1} - I) = \mathrm{tr}(B + \Lambda B^T \Lambda^{-1} + B\Lambda B^T\Lambda^{-1}) \le 3d \|B\|_{\max} \frac{\max(\lambda_1,\ldots,\lambda_d)}{\min(\lambda_1,\ldots,\lambda_d)} \le \epsilon^2.$$
Hence the first term on the right hand side of \eqref{tri} is bounded by $2\epsilon$. The last term of \eqref{tri} equals to
\begin{align*}
 & \left\|\sum_{h > H} \pi_h \phi_{\hat\Sigma}(\cdot - z_h) + \sum_{h = 1}^H \pi_h\{\phi_{\hat \Sigma}(\cdot - z_h) -\phi_{\hat\Sigma}(\cdot - \hat z_h)\}+ \sum_{h = 1}^H (\pi_h - \hat\pi_h) \phi_{\hat\Sigma}(\cdot - \hat z_h)\right\|_1\\
& ~~~~~~~ \le \sum_{h > H} \pi_h + \sum_{h = 1}^H \pi_h \|\phi_{\hat\Sigma}(\cdot - z_h) - \phi_{\hat\Sigma}(\cdot - \hat z_h)\|_1 + \sum_{h = 1}^H |\pi_h - \hat \pi_h|.
\end{align*}
The first term above is smaller than $\epsilon$ and so is the second term because
$$\|\phi(\cdot - z_h) - \phi(\cdot - \hat z_h)\|_1 \le \left(\frac2\pi\right)^{1/2} \|\hat\Sigma^{-1/2}(z_h - \hat z_h)\|\le \epsilon.$$
The last term is smaller than or equal to $(1 - \sum_{h > H}\pi_h)\sum_{h = 1}^H|\tilde \pi_h - \hat\pi_h| + \sum_{h > H}\pi_h\sum_{h = 1}^H \hat\pi_h \le 2\epsilon$.
Thus a $(6\epsilon)$-net of $\scQ$, in the $L_1$-topology, can be constructed with $\hat p = p_{\hat F, \hat \Sigma}$ as above. The total number of such $\hat p$ is bounded by a multiple of
$
\left\{a/(\sigma_0\epsilon)\right\}^{dH}{\epsilon}^{-H}{\delta}^{-d(d - 1)/2}M^d
$. This proves the first assertion with $\rho = \|\cdot\|_1$, because $M\log(1 + \epsilon^2 / d) \lesssim M\epsilon^2$ and the constant factor by 6 can be absorbed in the bound. The same obtains when $\rho$ is the Hellinger metric because it is bounded by the square-root of the $L_1$-metric.

For the second assertion, we know that a Dirichlet process $F \sim \mathcal{D}_{\alpha}$ can be represented by a Sethuraman's stick-breaking process as
\begin{equation}
\label{dp2}
F = \sum_{h = 1}^\infty \pi_h \delta_{Z_h}, \;\; \pi_h =  V_h \prod_{j < h}(1 - V_j),
\end{equation}
where $\delta_x$ is the Dirac measure at $x$, $\{V_h, h \ge 1\}$ are independent beta distributed random variables with parameters $1$
and $|\alpha| = \alpha(\mathrm{R}^d)$, $\{Z_h, h \geq 1\}$ are independently distributed according to the probability measure
$\bar \alpha = \alpha / |\alpha|$ and these two sets of random variables are mutually independent.
Hence $ p_{F, \Sigma} = \sum_{h = 1}^\infty \pi_h \phi_\Sigma(\cdot - Z_h)$ with $\pi_h$ and $Z_h$ as described in \eqref{dp2}.
Therefore, with $\Pi$ denoting the Dirichlet mixture prior of Section \ref{sec: prior},
\begin{align}
\Pi(\scQ^c)~~\le & ~~H \bar\alpha([-a,a]^d)^c + \pr\left(\sum_{h > H}\pi_h > \epsilon\right) + \pr\left\{\mathrm{eig}_d(\Sigma^{-1}) > \sigma_0^{-2}\right\} \nonumber \\
&~~~~~~~~~~ + \pr\left\{\mathrm{eig}_1(\Sigma^{-1}) \leq \sigma_0^{-2}\left(1 + \frac{\epsilon^2}{d}\right)^{-M}\right\}.
\label{pc1}
%&\le &He^{-a^2}  + e^{-\left(\frac{1}{\sigmalo}\right)^2} + \bigg(\frac{1}{\sigmalo(1 + \epsilon)^M}\bigg)^\nu + \Pr(\prod_{h \le H} (1 - V_h) > \epsilon)
\end{align}
The first term is bounded by $b_1H\exp(-C_1 a^{a_1})$ by assumption on $\alpha$. Because $W = -\sum_{h = 1}^H\log(1 - V_h)$ is gamma distributed with parameters $H$ and $|\alpha|$,
we have
$$\pr\left(\sum_{h > H}\pi_h > \epsilon\right) = \pr\left(W < \log\frac1  \epsilon\right) \le \frac{(-|\alpha|\log \epsilon)^H }{ \Gamma(H + 1)} \le \bigg(\frac{e|\alpha|}{H}\log\frac{1}{\epsilon}\bigg)^H$$
by Stirling's formula. The last two terms are bounded by a multiple of
$
b_2 \exp\{-C_2 \sigma_0^{-2a_2}\} + b_3\sigma_0^{-2a_3} \left(1 + \epsilon^2/d\right)^{- M a_3}
$. This proves the second assertion.
\end{proof}

\begin{proof}[of Theorem \ref{thm:x3}]
For any $\sigma > 0$, define the transformation $T_{\alpha, \beta, \sigma}$ on $\mathcal{C}^{\alpha, \beta, L, \tau_0}(\mathbb{R}^d)$ as
\begin{equation}
\label{eq:trans3}
T_{\alpha,\beta,\sigma}f = f - \sum_{{k \in \natz^d:\; 1 \le \langle k, \alpha \rangle < \beta}} d_k \sigma^{\langle k, \alpha \rangle} f.
\end{equation}
Also define $K_{\alpha, \sigma} f$ as the convolution of $f$ and the normal density with mean zero and variance $\diag(\sigma^{2\alpha_1}, \ldots, \sigma^{2\alpha_d})$. The anisotropic analog of Lemma \ref{lem:basic} is that there exists a constant $M_{\alpha, \beta}$ such that for any $f \in \mathcal{C}^{\alpha, \beta, L, \tau_0}$ and any $\sigma \in (0, 1/(2\tau_0)^{1/2\alpha{\max}})$, $|\{K_{\alpha, \sigma}(T_{\alpha,\beta, \sigma} f) - f\}(x)| < M_\beta L(x)\sigma^\beta$ for all $x \in \mathbb{R}^d$. This follows along the lines of our proof of Lemma \ref{lem:basic} starting from the anisotropic Taylor approximation
\[
f(x + y) - f(x) = \sum_{1 \le \langle k, \alpha \rangle < \beta} \frac{(-y)^k}{k!} (D^kf)(x) + R(x, y),
\]
where the residual $R(x, y)$ in absolute value is bounded by a sum over terms of the form
\begin{align*}
\frac{|y|^k}{k!} &|(D^kf)(x_1, \ldots, x_{j - 1}, x_j + \xi_j, x_{j + 1} + y_{j + 1}, \ldots, x_d + y_d)\\
& ~~~~~~~~~~ - (D^kf)(x_1, \ldots, x_{j - 1}, x_j, x_{j + 1} + y_{j + 1}, \ldots, x_d + y_d)|\\
& \le L(x)\exp(\tau_0 \|y\|_1^2) |y|^k |y_j|^{\min(\beta / \alpha_j - k_j, 1)} / k!
 \end{align*}
 with $j$ such that $\beta > \langle k, \alpha\rangle > \beta -  \alpha_j$. Consequently, $\int |R(x, y)| \phi_{\diag(\sigma^{2\alpha})}(y)dy \le K_1 L(x) \sigma^{\beta}$ for some constant $K_1$. The rest of the argument in our proof of Lemma \ref{lem:basic} goes through. The pointwise error bound between $f_0$ and $K_{\alpha,\sigma}(T_{\alpha,\beta,\sigma} f)$ then leads to exact analogs of Theorem \ref{thm:hell approx} and Proposition \ref{prop:compact}, giving a $\tilde h_\sigma$ with support inside $\{x \in \mathbb{R}^d: \|x\| \le a_0 \{\log (1/\sigma)\}^\tau\}$ satisfying $d_H(f_0, K_{\alpha,\sigma} \tilde h_\sigma) \le K_0 \sigma^\beta$ for some constant $K_0$. Next the arguments in the proof of Theorem \ref{thm:thick} can be replicated, with $\mathcal{P}_\sigma$ built around a discrete $F_\sigma = \sum_{j = 1}^N p_j \delta_{z_j}$ with $N \le D_1 \sigma^{-d} \{\log(1/\tilde \epsilon_n)\}^{d + d / \tau}$ support points such that $d_H(K_{\alpha,\sigma} \tilde h_\sigma, K_{\alpha, \sigma} F_\sigma) \le A_1\tilde\epsilon_n^{b_1} \{\log (1/\tilde \epsilon_n)\}^{1/4}$. We also need to define $\mathcal{S}_\sigma$ as the set of $\Sigma$ such that $\eig_j(\Sigma^{-1})$ lies between $\sigma^{-2\alpha_j}$ and $\sigma^{-2\alpha_j}(1 + \sigma^{2\beta})$ for each $j = 1, \ldots, d$. The prior probability of this set under $G$ is bounded from below by $C_3 \exp[-c_3\tilde\epsilon_n^{-\kappa\alphamax / \beta} \{\log(1 / \tilde \epsilon_n)\}^{s\kappa + 1}]$ which contributes the $\kappa\alphamax$ term in $\dstar = \max(d, \kappa\alphamax)$.
\end{proof}

\section*{Appendix B. Supplementary results}
\label{apndx2}

\begin{theorem}
\label{lem disc approx}
Let $P_0$ be a probability measure on $\{x \in \mathbb{R}^d: \|x\| \le a\} \subset \mathbb{R}^d$. For any $\varepsilon >0$ and $\sigma > 0$, there is a discrete probability measure $F_\sigma$ on $\{x \in \mathbb{R}^d: \|x\| \le a\}$ with at most $N_{\sigma, \varepsilon} = D[\{(a / \sigma) \vee 1\} \log(1/\varepsilon)]^d$ support points such that $\|p_{P_0, \sigma} - p_{F_\sigma, \sigma}\|_\infty \lesssim \varepsilon /\sigma^d$ and $\|p_{P_0, \sigma} - p_{F_\sigma, \sigma}\|_1 \lesssim \varepsilon \{\log (1/ \varepsilon)\}^{1/2}$, for some universal constant $D$.
\end{theorem}

\begin{proof} The proof is a straightforward extension of Lemma 2 of \citet{Ghosal2007} and Lemma 3.1 of \citet{Ghosal20011} to $d$ dimensions. For any probability distribution $F$ on $\mathbb{R}^d$, there exists a discrete distribution $F'$ with at most $\{(2k-2)^d + 1\}$ support points such that the mixed moments $z_1^{l_1} z_2^{l_2}\cdots z_d^{l_d}$ are matched up for every $1 \leq l_i \leq 2k-2$ $(i=1,\ldots,d)$. This power of $d$ propagate all through the require extensions and appears in $N_{\sigma, \epsilon}$ in the statement of the current theorem.

%The only subtlety lies in replacing display (3.9) of \citet{Ghosal20011} with
%\begin{equation}
%\int z^l dF(z) = \int z^l dF'(z), \;\;l \in \{1, \ldots, 2k - 2\}^d
%\label{new moment}
%\end{equation}
%where, for a $z = (z_1, \ldots, z_d) \in \mathbb{R}^d$ and a $l = (l_1, \ldots, l_d) \in \inte^d$, $z^l$ denotes $z_1^{l_1} z_2^{l_2}\ldots z_d^{l_d}$. For any probability distribution $F$ on $\mathbb{R}^d$, there exists a discrete distribution $F'$ with at most $\{2(k - 1)\}^d + 1$ support points, satisfying (\ref{new moment}). This power of $d$ propagate all through the require extensions and appears in $N_{\sigma, \epsilon}$ in the statement of the current theorem.
\end{proof}

\begin{corollary}
\label{cor:disc approx}
Let $P_0$ be a probability measure on $\{x \in \mathbb{R}^d: \|x\| \le a\}$. For any $\varepsilon >0$ and $\sigma > 0$, there is a discrete probability measure $F^*_\sigma$ on $\{x \in \mathbb{R}^d: \|x\| \le a\}$ with at most $N_{\sigma, \varepsilon} = D[\{(a / \sigma) \vee 1\} \log(1/\varepsilon)]^d$ support points from the set $\{(n_1, \ldots, n_p) \sigma\varepsilon : n_i \in \inte, |n_i| < \lceil a/(\sigma\varepsilon)\rceil, i = 1, \ldots, p\}$
such that $\|p_{P_0, \sigma} - p_{F^*_\sigma, \sigma}\|_\infty \lesssim \varepsilon /\sigma^{d}$ and $\|p_{P_0, \sigma} - p_{F^*_\sigma, \sigma}\|_1 \lesssim \varepsilon \{\log (1/ \varepsilon)\}^{1/2}$.
\end{corollary}

\begin{proof}
First get $F_{\sigma}$ as in Theorem \ref{lem disc approx} and then move each of its support points to the nearest point on the grid $\{(n_1, \ldots, n_p) \sigma\varepsilon : n_i \in \inte, |n_i| < \lceil a/(\sigma\varepsilon)\rceil, i = 1, \ldots, p\}$ to get $F^*_\sigma$. These moves cost at most a constant times $\epsilon^2/\sigma^d$ to the supremum norm distance and at most a constant times $\epsilon$ to the $L_1$ distance.
\end{proof}

\begin{lemma}
\label{lem:part approx}
Let $V_0, V_1, \ldots, V_N$ be a partition of $\mathbb{R}^d$ and $F' = \sum_{j = 1}^N p_j \delta_{z_j}$ a probability measure on $\mathbb{R}^d$ with $z_j \in V_j$, $j = 1, \ldots, N$. Then, for any probability measure $F$ on $\mathbb{R}^d$, and any $\sigma > 0$,
\begin{align*}
\|p_{F, \sigma} - p_{F', \sigma}\|_\infty & \lesssim \frac{1}{\sigma^{d + 1}}\max_{1 \le j \le N} \diam(V_j) + \frac{1}{\sigma^d}\sum_{j = 1}^N |F(V_j) - p_j|,\\
\|p_{F, \sigma} - p_{F', \sigma}\|_1 & \lesssim \frac{1}{\sigma}\max_{1 \le j \le N} \diam(V_j) + \sum_{j = 1}^N |F(V_j) - p_j|,
\end{align*}
where $\diam(A) = \sup\{\|z_1 - z_2\| : z_1, z_2 \in A\}$ denotes the diameter of a set $A$.
\end{lemma}

\begin{proof} The proof is an extension of Lemma 5 of \cite{Ghosal2007} to $d$ dimensions.
\end{proof}

%\begin{lemma}[Lemma 10 of \cite{Ghosal2007}]
%Let $(X_1, \ldots, X_N) \sim {\text Dir}(\alp%ha_1, \ldots, \alpha_N)$, with $0 < \alpha_j \le 1$, $\sum_{i = 1}^N \alpha_j = m$. Fix $a > 0$, $b > 0$. Then, there exist constants $c$ and $C$ that only depend $a$, $b$ and $m$ such that for any $\varepsilon \in (0, \min(1/4, a \{\min_j\alpha_j\}^b, 1 / N))$,
%$$P\left(\sum_{j = 1}^N |X_j - p_j| \le 2\varepsilon, \min_j X_j \ge \varepsilon^2 / 2\right) \ge C \exp\left(-c N \log \frac{1}{\varepsilon}\right)$$
%\end{lemma}

\begin{lemma}
\label{lem:kl by hell}
There is a $\lambda_0 \in (0, 1)$ such that for any two probability measures $P, Q$ with densities $p, q$ and any $\lambda \in (0, \lambda_0)$
\begin{align*}
P\log\frac pq & \le  d_H^2(p, q)\left(1 + 2\log\frac1\lambda\right) + 2P\left\{\left(\log\frac pq\right) \Ind\left(\frac qp \le \lambda\right)\right\},\\
P\left(\log\frac pq\right)^2 & \le  d_H^2(p, q)\left\{12 + 2\left(\log\frac1\lambda\right)^2\right\} + 8P\left\{\left(\log\frac pq\right)^2 \Ind\left(\frac qp \le \lambda\right)\right\}.
\end{align*}
\end{lemma}

\begin{proof}
Our proof follows the argument presented in the proof of Lemma 7 of \citet{Ghosal2007}. The function $r:(0, \infty)\to \mathbb{R}$ defined implicitly by $\log x = 2(x^{1/2} - 1) - r(x)(x^{1/2} - 1)^2$ is non-negative and decreasing, and there exists a $\lambda_0 > 0$  such that $r(x) \le 2 \log(1/x)$ for all $x \in (0, \lambda_0)$. Using these properties and $d_H^2(p,q) = -2P\{(q/p)^{1/2} - 1\}$ we obtain
\begin{align*}
P\log\frac pq & = d_H^2(p, q) + P\left\{r\left(\frac qp\right)\left(\frac{q^{1/2}}{p^{1/2}} - 1\right)^2\right\}\\
& \le 	d_H^2(p, q) + r(\lambda)d_H^2(p, q) + P\left\{r\left(\frac qp\right)\Ind\left(\frac{q}{p}< \lambda\right)\right\} \\
& \le d_H^2(p, q) + 2\left(\log\frac1\lambda\right) d_H^2(p, q) + 2P\left\{\left(\log\frac pq\right) \Ind\left(\frac{q}{p}< \lambda\right)\right\}
\end{align*}
for any $\lambda < \lambda_0$, proving the first inequality of the Lemma.

To prove the other inequality, note that $|\log x|\le 2|x^{1/2} - 1|$, $x \ge 1$ and so
\[
P\left\{ \left(\log \frac pq\right)^2 \Ind\left(\frac qp \ge 1\right)\right\} \le 4 P\left( \frac{q^{1/2}}{p^{1/2}} - 1\right)^2 = 4d_H^2(p, q).
\]
On the other hand,
\begin{align*}
P\left\{ \left(\log \frac pq\right)^2\right. & \left.  \Ind\left(\frac qp \le 1\right)\right\}   \\
& \le 8 P\left( \frac{q^{1/2}}{p^{1/2}} - 1\right)^2 + 2P\left\{ r^2\left(\frac qp\right)\left(\frac{q^{1/2}}{p^{1/2}} - 1\right)^4 \Ind\left(\frac qp \le 1\right)\right\}\\
& \le 8d_H^2(p, q) + 2r^2(\lambda)P\left(\frac{q^{1/2}}{p^{1/2}} - 1\right)^2 + 2P\left\{r^2\left(\frac qp\right)\Ind\left(\frac qp \le \lambda\right)  \right\}\\
& \le 8d_H^2(p,q) + 2\left(\log\frac1\lambda\right)^2d_H^2(p,q) + 8P\left\{\left(\log \frac pq\right)^2\Ind\left(\frac qp \le \lambda\right)\right\}.
\end{align*}
This completes the proof.
%giving us the desired result.
\end{proof}

\begin{lemma}
\label{lem:hell mix}
Let $\mathcal{A}, \mathcal{X}$ be metric spaces and suppose $\{p_\alpha\}_{\alpha \in \mathcal{A}}$ and $\{q_\alpha\}_{\alpha \in \mathcal{A}}$ are collections of probability density functions on $\mathcal{X}$ with respect to a dominating measure $\nu$. Then for any probability measure $G$ on $\mathcal{A}$, $d^2_H(\int p_\alpha dG, \int q_\alpha dG) \le \int d^2_H(p_\alpha, q_\alpha)dG$. In particular, for any three densities $p, q$ and $\phi$ on $\mathbb{R}^d$, $d_H(\phi * p, \phi * q) \le d_H(p, q)$.
\end{lemma}

\begin{proof}
By the Cauchy--Schwartz inequality, $1 - \int d^2_H(p_\alpha, q_\alpha)dG /2$ equals to
\begin{align*}
 \int\int \{p_\alpha(x) q_\alpha(x)\}^{1/2}\nu(dx)G(d\alpha) & \le \int \left\{\int p_\alpha(x) G(d\alpha) \int q_\alpha(x) G(d\alpha)\right\}^{1/2} \nu(dx),
  \end{align*}
 which is the same as $1 - \frac12 d_H^2\left(\smallint\hspace{-.03in} p_\alpha dG, \smallint\hspace{-.03in} q_\alpha dG\right)$.
This gives the first result. The second assertion holds by choosing $\mathcal{A} = \mathcal{X} = \mathbb{R}^d$, $p_\alpha(x) = p(x- \alpha)$, $q_\alpha(x) = q(x- \alpha)$ and $G(d\alpha) = \phi(\alpha)d\alpha$.
\end{proof}

\begin{lemma}
\label{rm1}
Suppose a probability density function $f_0$ satisfies the tail condition \eqref{eq:tail}, $\log f_0 \in  \mathcal{C}^{\beta, Q_1, 0}(\mathbb{R}^d)$ for some polynomial $Q_1$ with $P_0|D^k \log f_0|^{{(2\beta + \epsilon)}/{k_\cdot}} < \infty$, $k\in \natz^d$, $k_\cdot \le \lfloor\beta\rfloor$ and $P_0Q_1^{{(2\beta + \epsilon)}/{\beta}}  < \infty$. Additionally, suppose
\begin{equation}
\label{ratio}
\left|\frac{f_0(x+y) }{ f_0(x)} - 1\right| \leq  Q(x) e^{\tau_0 \|y\|^2}\|y\|^{\beta - \lfloor\beta\rfloor},~\mbox{for any}~x,y \in \mathbb{R}^d,
\end{equation}
for some $\tau_1 > 0$ and a function $Q$ satisfying $P_0Q^2 < \infty$. Then, there exist a $\tau_0 > 0$ and a positive functions $L(x)$ such that $f_0 \in \mathcal{C}^{\beta, L, \tau_0}(\mathbb{R}^d)$ and \eqref{eq:int} holds.
\end{lemma}

Without \eqref{ratio}, the assumptions made on $f_0$ in the above lemma match one to one with conditions C1-C3 of \citet{Kruijer2010}. The additional assumption \eqref{ratio} is a mild one and is satisfied by densities with tails exactly as in the bound \eqref{eq:tail} with $\tau \le 2$, and also by finite mixtures of such densities.

%The $L$ in \eqref{eq:int} is the true density multiplied by a polynomial. This is equivalent to (C2) of \citet{Kruijer2010}.
%By the mean value theorem, for $f > 0$,
%$|f(x+y)-f(x)|/f(x)=|\exp\{\log f(x+y)-\log f(x)\}-1| \leq |\log f(x+y)-\log f(x)|\{1+f(x+y)/f(x)\}$,
%so what we need in \eqref{eq:int} is an upper bound to $f(x+y)/f(x)$ of the form $Q(x)\exp\{\tau_0 \|y\|^2\}$
%for some $Q$ and small $\epsilon > 0$ satisfying $\int f |Q|^{2+\epsilon} < \infty$,
%which is satisfied by many densities with tails no sharper than the normal and their finite mixtures.

\begin{proof}[of Lemma \ref{rm1}] For a multi-index $k \in\natz^d$, let $\mathcal{P}$ denote the set of all solutions $\{m^{(1)}, \ldots, m^{(q)}\}$
 to $k = m^{(1)} + \cdots + m^{(q)}$, $q \ge 1$, $m^{(j)} \in \natz^d$ with $m^{(j)}_\cdot \ge 1$ $(j = 1, \ldots, q)$. 
 Existence of $D^k f_0$ of all orders $k_\cdot \le \lfloor \beta\rfloor$ follows from the same property of $\log f_0$. In fact, by chain rule $D^kf_0(x) = f_0(x)\sum_{P \in \mathcal{P}(k)} \prod_{m \in P} D^m \log f_0(x)$ and so $P_0 |(D^k f_0)/f_0|^{(2\beta + \epsilon)/ k_\cdot} < \infty$ by an application of the H\"older inequality. Also, because $\log f_0 \in \mathcal{C}^{\beta, Q_1, 0}(\mathbb{R}^d)$ with $Q_1$ a polynomial, for every $k \in \natz^{d}$ with $k_\cdot < \beta$, we can find polynomial $Q_{k,1}$ and $Q_{k,2}$ such that $|D^k \log f_0(x)| < Q_{k,1}(x)$ and $|D^k \log f_0(x + y) - D^k \log f_0(x)| < Q_{k,2}(x)e^{\|y\|^2}\|y\|^{\beta - \lfloor\beta\rfloor}$. Hence, for $k_\cdot = \lfloor\beta\rfloor$, $|D^k f_0(x + y) - D^k f_0(x)|$ can be bounded by $|f_0(x + y) - f_0(x)|Q_3(x) + f_0(x)Q_4(x)e^{\tau_2\|y\|^2}\|y\|^{\beta - \lfloor\beta\rfloor}$ for some polynomials $Q_3$ and $Q_4$ and a $\tau_2 > 0$. Therefore $f_0 \in \mathcal{C}^{\beta, L, \tau_0}$ for $\tau_0 = \max(\tau_1, \tau_2)$ and $L(x) = f_0(x)\{Q(x)Q_3(x) + Q_4(x)\}$. Because of the tail condition on $f_0$, for any polynomial $\tilde Q$ and $a > 0$, $P_0 |\tilde Q|^a < \infty$. And so $P_0(L/f_0)^{2 + \epsilon / \beta } < \infty$ by H\"older's inequality and the assumption on $Q$.
\end{proof}

%We give a sketch of the proof. For any $a=(a_1,\ldots,a_d), b=(b_1,\ldots,b_d) \in \mathbb{N}_0^d$, define $a \leq b$ if $a_i \leq b_i$ $(i=1,\ldots,d)$.
%For every $k \in \mathbb{N}_0^d$ with $k. \leq \lfloor \beta \rfloor$, $D^k f$ is finite by the chain rule
% because $D^j (\log f)$ is finite for any $j \leq k$. In order to verify \eqref{smoothness} for $k.=\lfloor \beta \rfloor$, we show that $|D^j \log f(x)| \leq P_1(x)$ and $|D^j\log f(x+y) - D^{j} \log f(x) | \leq P_2(x) \|y\| \exp\{\tau_0 \|y\|^2\} $ for some polynomials
%  $P_1$, $P_2$ and any $j \leq k$ with $j. < \lfloor \beta \rfloor$.
%  By the chain rule, $(D^k f)(x) = f(x) \tilde{P}\{D^j \log f(x): j. > 0, j \leq k \}$, where $\tilde{P}$ is a polynomial without the constant term. Therefore $\left|(D^kf)(x + y) - (D^kf)(x)\right|$ is bounded by
%$|f(x+y) - f(x)|$ times a polynomial in $\left\{D^j \log f(x): j. > 0, j \leq k \right\}$ plus
% $f(x)$ times a polynomial in  $\left\{D^j \log f(x+y) - D^j \log f(x),  D^j \log f(x+y), D^j \log f(x): j. > 0, j \leq k \right\}$.
%Both terms are bounded by $f(x) Q_3(x) \exp\{\tau_1 \|y\|^2\} \|y\|^{\beta-\lfloor \beta \rfloor  }$ for some constants $C,\tau_1 > 0$.

\bibliography{ref}
\bibliographystyle{biometrika}

\end{document}